\documentclass[a4paper,12pt]{article}

\usepackage[english]{babel}
\usepackage{hyperref}
\usepackage{amssymb,amsmath,amsthm}

\usepackage{enumerate,enumitem}
\usepackage{graphicx,color}
\usepackage{epsfig}
\usepackage{tikz}
\usetikzlibrary{calc,decorations.pathmorphing,decorations.text,decorations.markings,matrix}
\usepackage[font=footnotesize]{caption}
\usepackage[font=footnotesize]{subcaption}
\usepackage{refcount}
\usepackage[hmargin=2.5cm,vmargin=2.5cm]{geometry}
\usepackage{mathtools, braket,bm}
\usepackage[normalem]{ulem}
\usepackage{authblk}
\usepackage{xcolor} 
\usepackage[square,sort,comma,numbers]{natbib}
\usepackage[thicklines]{cancel}
\usepackage{centernot}
\usepackage{indentfirst}
\usepackage[nameinlink]{cleveref}
\usepackage{booktabs}

\usepackage{array}
\allowdisplaybreaks[2]

\renewcommand{\ge}{\geqslant}
\renewcommand{\geq}{\geqslant}
\renewcommand{\le}{\leqslant}
\renewcommand{\leq}{\leqslant}

\theoremstyle{plain}
\newtheorem{thm}{Thm}[section]

\newtheorem{theorem}[thm]{Theorem}
\newtheorem{lemma}[thm]{Lemma}
\newtheorem{corollary}[thm]{Corollary}
\newtheorem{proposition}[thm]{Proposition}
\newtheorem{conjecture}[thm]{Conjecture}

\newtheorem{observation}[thm]{Observation}
\newtheorem{definition}[thm]{Definition}

\newtheorem*{Mtheorem}{Main Theorem}
\newtheorem*{Stheorem}{Structural Theorem}

\def\aftermath{\par\vspace{-\belowdisplayskip}\vspace{-\parskip}\vspace{-\baselineskip}}

\newenvironment{proof*}{\noindent \emph{Proof.}}{\hfill$\Diamond$}

\def\Z{\mathbb{Z}}

\def\P{\mathcal{P}}
\def\Q{\mathcal{Q}}

\def\N{\mathcal{N}}
\def\S{\mathcal{S}}

\def\SZ{\mathrm{S}\Z}
\def\SZZ#1{\SZ_{#1}}  

\crefname{theorem}{Theorem}{Theorems}
\Crefname{theorem}{Theorem}{Theorems}

\begin{document}
	\title{Orientations of $10$-Edge-Connected\\ Planar Multigraphs and Applications}
	\date{}
	
	\author{{\normalsize Daniel W. Cranston\thanks{William \& Mary, Department of Mathematics; Williamsburg, VA, USA; \texttt{dcransto@gmail.com}}, 
	Jiaao Li\thanks{School of Mathematical Sciences and LPMC, Nankai University, Tianjin 300071, China; \texttt{lijiaao@nankai.edu.cn}; Research supported by National Natural Science Foundation of China (No. 12571371) and Natural Science Foundation of Tianjin (No. 24JCJQJC00130)}, 
			Bo Su\thanks{School of Mathematical Sciences and LPMC, Nankai University, Tianjin 300071, China; \texttt{suboll@mail.nankai.edu.cn}}, 
			Zhouningxin Wang\thanks{School of Mathematical Sciences and LPMC, Nankai University, Tianjin 300071, China; \texttt{wangzhou@nankai.edu.cn}; Research supported by National Natural Science Foundation of China (No. 12301444).}, and Chunyan Wei\thanks{Department of Mathematical and Physical Sciences, Henan University of Engineering, Zhengzhou 451191, China; \texttt{yan1307015@163.com}}}}
	
	\maketitle
	
\begin{abstract}
A graph is called strongly $\Z_{2k+1}$-connected if for each boundary function $\beta: V(G)\mapsto \Z_{2k+1}$ with $\sum_{v\in V(G)}\beta(v)\equiv 0\pmod{2k+1}$, there exists an orientation $D$ of $G$ such that $d_D^+(v) - d_D^-(v) \equiv \beta(v) \pmod{2k+1}$ for each $v \in V(G)$. 

We show that every planar multigraph with $5$ edge-disjoint spanning trees is strongly $\Z_{5}$-connected. This verifies a special case of the Additive Base Conjecture when restricted to planar graphs. Hence, every $10$-edge-connected directed planar graph admits an antisymmetric $\Z_5$-flow. By duality, every orientation of a planar graph of girth at least $10$ admits a homomorphism to a $5$-vertex tournament.  

Our result also gives a new proof of the known result that every planar graph of girth at least $10$ has a homomorphism to the $5$-cycle.
\end{abstract}


\section{Introduction}
Graphs in this paper are allowed to have parallel edges but no loops. We generally adopt terminology and notation from the books of Bondy and Murty~\cite{BM08} and Zhang~\cite{Z97}.
	
\subsection{Circular Flows and Antisymmetric Flows}
	
In 1984, Jaeger~\cite{J84} introduced the concept of circular flows in graphs. 
But the definition most widely used today, which is equivalent, was formulated in 1998 by Goddyn et al.~\cite{GTZ98}; they presented circular flows as a dual notion of circular colorings. This concept naturally extends the classical notion of integer flows due to Tutte~\cite{T54}.

Let $G$ be a graph and let $p$ and $q$ be integers with $p \geqslant 2q > 0$. 
A \emph{circular $\frac{p}{q}$-flow} of $G$ is a flow $(D,f)$ such that
every edge $e \in E(G)$ satisfies
\[
q \leqslant |f(e)| \leqslant p - q.
\]

In 1981, Jaeger~\cite{J84,J88} conjectured (in an equivalent form) that every $4k$-edge-connected graph admits a circular $(2+\frac{1}{k})$-flow. In the planar setting, this conjecture reduces to the following Planar Circular Flow Conjecture.

\begin{conjecture}[Planar Circular Flow Conjecture~\cite{J84,J88}]
\label{conj:planar circular flow}
Every $4k$-edge-connected planar graph admits a circular $(2+\frac{1}{k})$-flow.
\end{conjecture}

The $k=1$ case of \Cref{conj:planar circular flow} is the dual of the famous Gr\"{o}tzsch's Theorem: every triangle-free planar graph is $3$-colorable. The present paper focuses on the first open case of \Cref{conj:planar circular flow}, namely $k=2$. 
Lov\'{a}sz, Thomassen, Wu, and Zhang~\cite{LTWZ13} proved that every $6k$-edge-connected graph admits a circular $(2+\frac{1}{k})$-flow; in particular this guarantees a circular $\frac{5}{2}$-flow for each $12$-edge-connected graph. 
By also assuming planarity, Cranston and Li~\cite{CL20} proved the same conclusion with a weaker requirement on connectivity; see \Cref{thm:10-planar-circular-flow}. 
The corresponding dual result for circular colorings was previously proved by Dvo\v{r}\'{a}k and Postle~\cite{DP17}: every planar graph of girth at least $10$ admits a homomorphism to a cycle of length $5$.

\begin{theorem}[\cite{CL20,DP17}]
\label{thm:10-planar-circular-flow}
Every $10$-edge-connected planar graph admits a circular $\frac{5}{2}$-flow.
\end{theorem}

In this paper, we strengthen \Cref{thm:10-planar-circular-flow} using the notion of group connectivity. 
We state this result below in \Cref{thm:10-planar-SZ5}, after providing the needed definitions.
As an immediate corollary of \Cref{thm:10-planar-SZ5}, we recover \Cref{thm:10-planar-circular-flow}.
But more importantly, our results also contribute to the study of antisymmetric flows in directed graphs.

Let $\Gamma$ be an abelian group and let $\vec{G}=(V,\vec{E})$ be a directed graph. 
A $\Gamma$-flow $f\colon \vec{E}\to \Gamma$ is an \emph{antisymmetric flow} of $\vec{G}$ (abbreviated as a $\Gamma$-ASF) if no two arcs are assigned inverse elements of $\Gamma$. 
In particular, this condition implies that no arc is assigned the value $0$. Throughout this paper, a directed graph $\vec{G}$ is called \emph{$k$-edge-connected} if its underlying undirected graph, $G$, is $k$-edge-connected.

Many results on antisymmetric flows have been shown under various edge-connectivity hypotheses. 
Dvo\v{r}\'{a}k, Kaiser, Kr\'{a}{l}’, and Sereni~\cite{DKKS10} proved that every $3$-edge-connected directed graph admits a $\Z_2^3 \times \Z_3^9$-ASF. 
DeVos, Ne\v{s}et\v{r}il, and Raspaud~\cite{DNR04} proved that every $4$-edge-connected directed graph admits a $\Z_2^2 \times \Z_3^4$-ASF; every $5$-edge-connected directed graph admits a $\Z_3^5$-ASF; and every $6$-edge-connected directed graph admits a $\Z_2 \times \Z_3^2$-ASF. 
Esperet, de Joannis de Verclos, Le, and Thomass\'{e}~\cite{EJLT18} proved that every $12$-edge-connected directed graph admits a $\Z_5$-ASF. Cranston and Li~\cite{CL20} further showed that every $11$-edge-connected directed planar graph admits a $\Z_5$-ASF. 
We strengthen this result as follows.

\begin{theorem}
\label{thm:10-directed-planar-Z5}
Every $10$-edge-connected directed planar graph admits a $\Z_5$-ASF.
\end{theorem}

The dual formulation of \Cref{thm:10-directed-planar-Z5} asserts that every directed planar graph of girth at least $10$ admits a homomorphism to a directed simple graph on at most $5$ vertices. In other words, every planar graph of girth at least $10$ has oriented chromatic number at most $5$. Similar to \Cref{thm:10-planar-circular-flow}, \Cref{thm:10-directed-planar-Z5} follows directly from \Cref{thm:10-planar-SZ5}, below.

\subsection{Group Connectivity and Additive Bases}

Let $G$ be a graph and let $D$ be an orientation of $G$. 
For each $v \in V(G)$, let $E_D^+(v)$ (resp., $E_D^-(v)$) denote the set of arcs in $D$ with tail (resp., head) at $v$, and let $d_D^+(v) := |E_D^+(v)|$ and $d_D^-(v) := |E_D^-(v)|$. When $D$ is clear from context, we omit the subscript.

Given a graph $G$ and a positive integer $k$, an orientation $D$ of $G$ is a \emph{modulo $(2k+1)$-orientation} if
\(
d_D^+(v) - d_D^-(v) \equiv 0 \pmod{2k+1}
\)
for every $v \in V(G)$. 
In 1988, Jaeger~\cite{J88} observed that a graph $G$ admits a circular $\left(2+\frac{1}{k}\right)$-flow if and only if $G$ admits a modulo $(2k+1)$-orientation. 
In 2014, the concept of strong $\mathbb{Z}_{2k+1}$-connectivity was studied in~\cite{LLLMMSZ14}, for which a modulo $(2k+1)$-orientation is a special case.
For each odd integer $m\ge3$, a mapping $\beta:V(G)\to \Z_m$ is a {\it $\Z_m$-boundary} of $G$ if $\sum_{v\in V(G)}\beta(v)\equiv0~\pmod{m}$. Given a $\Z_m$-boundary $\beta$ of $G$, an orientation $D$ of $G$ is a \emph{$\beta$-orientation} if $d_D^+(v)-d_D^-(v)\equiv\beta(v)~\pmod{m}$ for every $v\in V(G)$.

\begin{definition}[{\cite{LLLMMSZ14}}]\label{def:SZk}
Let $m$ be an odd integer with $m\geq 3$. A graph $G$ is \emph{strongly $\Z_{m}$-connected} if for every $\Z_{m}$-boundary $\beta$ of $G$, there exists a $\beta$-orientation.
\end{definition}

When a graph $G$ is strongly $\Z_{m}$-connected, we simply write $G \in \SZ_{m}$. Cranston and Li~\cite{CL20} showed that $G\in\SZ_5$ whenever $G$ is an $11$-edge-connected planar graph.
In this paper, we strengthen their result as follows, which yields several consequences for circular flows and antisymmetric flows; see \Cref{thm:10-planar-circular-flow} and \Cref{thm:10-directed-planar-Z5}.

\begin{theorem}\label{thm:10-planar-SZ5}
Every $10$-edge-connected planar graph is strongly $\Z_5$-connected.
\end{theorem}

In fact, we prove the following stronger result; see~\Cref{sec:proof-main-results} for a detailed proof.

\begin{theorem}\label{thm:5treesSZ5}
Every planar graph containing five edge-disjoint spanning trees is strongly $\Z_5$-connected.
\end{theorem}

The Additive Base Conjecture of Jaeger, Linial, Payan, and Tarsi \cite{JLPT92} states, for each prime $p$, that the multiset union of the vectors of any family of $p$ linear bases of $\Z_p^n$ forms an additive basis of $\Z_p^n$ (i.e.~each element of $\Z_p^n$ can be expressed as a linear combination of these vectors with each coefficient $0$ or $1$). 
In graphs, each linear base is a spanning tree. So the graph version of the Additive Base Conjecture states that every graph with $p$ edge-disjoint spanning trees is strongly $\Z_p$-connected. 

A weak version of the $p=3$ case of the Additive Base Conjecture was recently proved by Y.~Yu~in an arXiv note (\href{https://doi.org/10.48550/arXiv.2510.01300}{arXiv 2510.01300}). He showed that the multiset union of the vectors of any family of four linear bases (instead of $3$, as conjectured) of $\Z_3^n$ contains an additive basis of $\Z_3^n$. 
In particular, this implies that every graph with four edge-disjoint spanning trees is strongly $\Z_3$-connected. 
For planar graphs, it is trivial to observe that having three edge-disjoint spanning trees implies being strongly $\Z_3$-connected; this is because every such planar graph must contain parallel edges, which form a simple reducible configuration for being strongly $\Z_3$-connected. 
Our \Cref{thm:5treesSZ5} verifies the first non-trivial case of the graph version of the Additive Base Conjecture on planar graphs.

By \Cref{thm:10-planar-SZ5}, every $10$-edge-connected planar graph admits a modulo $5$-orientation and hence a circular $\frac{5}{2}$-flow; thus, \Cref{thm:10-planar-SZ5} implies \Cref{thm:10-planar-circular-flow}. 
Moreover, Esperet, de Verclos, Le, and Thomass\'{e}~\cite{EJLT18} showed that for every graph $G \in \SZ_5$, every orientation $D$ of $G$ admits a $\Z_5$-ASF. 
Thus, \Cref{thm:10-planar-SZ5} also implies \Cref{thm:10-directed-planar-Z5}.

\subsection{Vertex Partitions and the Main Structural Theorem}
Let $G$ be a graph. 
A \emph{vertex partition} of $G$ is a collection $\{V_1, \dots, V_t\}$ such that
\[
\bigcup_{i=1}^t V_i = V(G)
\quad \text{and} \quad
V_i \cap V_j = \emptyset \ \text{for all distinct } i,j.
\]
If we let $\mathcal{P} := \{V_1,\ldots, V_t\}$, then
each set $V_i$ is a \emph{part} of $\mathcal{P}$, and the number of parts is denoted by $|\mathcal{P}|$. The edges of the induced subgraphs $G[V_1], \dots, G[V_t]$ are \emph{$\mathcal{P}$-internal edges}; all remaining edges of $G$ are \emph{$\mathcal{P}$-external edges}.

Given an edge $e \in E(G)$, \emph{contracting} $e$ means identifying the two endpoints of $e$ and deleting the resulting loop; the resulting graph is denoted by $G/e$. 
More generally, if $H \subseteq G$, then contracting all edges of $H$ yields the graph $G/H$. 
The graph $G/\mathcal{P}$ is formed from $G$ by identifying all vertices within each part of $\mathcal{P}$ into a single vertex (so $G/\mathcal{P}$ has $|\mathcal{P}|$ vertices) and deleting every resulting loop. 
To \emph{lift} a path $P = v_0 v_1 \dots v_n$ in a graph $G$, we delete all edges of $P$ and add a new edge $v_0 v_n$.

For each integer $a \ge 1$, let $aH$ denote the graph formed from $H$ by replacing each edge of $H$ with $a$ parallel edges, as shown in \Cref{fig:aK2}. 
Denote by $T_{a,b,c}$ (where $a,b,c \ge 0$) the graph on three vertices in which the three pairs of vertices are joined by $a$, $b$, and $c$ parallel edges, as shown in \Cref{fig:Tabc}. 
Finally, let $W_1$ and $W_2$ be two special planar graphs, each with four vertices and $12$ edges, shown in \Cref{fig:E1} and \Cref{fig:E2}.

\begin{figure}[!htbp]
    \centering
	\begin{subfigure}[t]{.24\textwidth}
		\centering
		\begin{tikzpicture}[scale=0.4]		
			\draw [line width=1pt, dotted] (0, 0.3) to (0, -0.3);
			\draw(0.8, -1.4) node[left=0.5mm]  {$a$};
			\draw [bend left=18, line width=0.6pt, black] (-3,0) to (3,0);
			\draw [bend right=18, line width=0.6pt, black] (-3,0) to (3,0);
			\draw [bend left=32, line width=0.6pt, black] (-3,0) to (3,0);
			\draw [bend right=32, line width=0.6pt, black] (-3,0) to (3,0);
				
			\draw [fill=white,line width=0.6pt] (-3,0) node[left=0.5mm] { } circle (7pt);  
			\draw [fill=white,line width=0.6pt] (3,0) node[right=0.5mm] { } circle (7pt);  
			\end{tikzpicture}
			\caption{$aK_2$}
			\label{fig:aK2}     
		\end{subfigure}
		\begin{subfigure}[t]{.24\textwidth}
			\centering
			\begin{tikzpicture}[scale=0.4]			
				\draw [rotate=240] [line width=1pt, dotted] (1, -0.6) to (1, -1.2);
				\draw [bend left=20, line width=0.6pt, black] (0,2) to  (-2.5,-2.5) ;
				\draw [bend right=18, line width=0.6pt, black] (0,2) to  (-2.5,-2.5) ;
				\draw [bend left=34, line width=0.6pt, black] (0,2) to  (-2.5,-2.5) ;
				\draw [bend right=32, line width=0.6pt, black] (0,2) to  (-2.5,-2.5) ;

				\draw [line width=1pt, dotted] (0, -2.2) to (0, -2.8);
				\draw [bend left=20, line width=0.6pt, black]  (-2.5,-2.5)  to (2.5,-2.5);
				\draw [bend right=18, line width=0.6pt, black]  (-2.5,-2.5)  to (2.5,-2.5);
				\draw [bend left=34, line width=0.6pt, black]  (-2.5,-2.5)  to (2.5,-2.5);
				\draw [bend right=32, line width=0.6pt, black]  (-2.5,-2.5)  to (2.5,-2.5);

				\draw [rotate=120] [line width=1pt, dotted](-1, -0.7) to (-1, -1.3);
				\draw [bend left=20, line width=0.6pt, black] (0,2) to (2.5,-2.5); 
				\draw [bend right=18, line width=0.6pt, black] (0,2) to (2.5,-2.5);
				\draw [bend left=34, line width=0.6pt, black] (0,2) to (2.5,-2.5);
				\draw [bend right=32, line width=0.6pt, black] (0,2) to (2.5,-2.5);
				
				\draw [fill=white,line width=0.6pt] (0,2) node[above=0.5mm] { } circle (8pt); 
				\draw [fill=white,line width=0.6pt] (2.5,-2.5) node[right=0.5mm] { } circle (8pt);     
				\draw [fill=white,line width=0.6pt] (-2.5,-2.5) node[left=0.5mm] { } circle (8pt);   
				\draw  (-2,0.2) node[left=0.5mm] {$a$};   
				\draw(0.8,-3.8) node[left=0.5mm] {$c$};  
				\draw  (3.3,0.2) node[left=0.5mm] {$b$};  
			\end{tikzpicture}
			\caption{$T_{a,b,c}$}
			\label{fig:Tabc}
		\end{subfigure} 
    \begin{subfigure}[t]{0.2\textwidth}
        \centering
\begin{tikzpicture}[scale=0.6]			
    \draw [line width=0.6pt, black] (0,3) to (-1.73,0); 
    \draw [line width=0.6pt, black] (0,3) to (1.73,0); 
    \draw [line width=0.6pt, black] (-1.73,0) to (1.73,0); 

    \draw [line width=0.6pt, black] (0,1.14) to (1.73,0); 
    \draw [bend left=18,line width=0.6pt, black] (0,1.14) to (1.73,0); 
    \draw [bend right=18,line width=0.6pt, black] (0,1.14) to (1.73,0); 

    \draw [line width=0.6pt, black] (0,1.14) to (-1.73,0); 
    \draw [bend left=18,line width=0.6pt, black] (0,1.14) to (-1.73,0); 
    \draw [bend right=18,line width=0.6pt, black] (0,1.14) to (-1.73,0); 

    \draw [line width=0.6pt, black] (0,1.14) to (0,3); 
    \draw [bend left=18,line width=0.6pt, black] (0,1.14) to (0,3); 
    \draw [bend right=18,line width=0.6pt, black] (0,1.14) to (0,3); 

    \draw [fill=white,line width=0.6pt] (0,3) node[above] {} circle (5pt) ; 
    \draw [fill=white,line width=0.6pt] (-1.73,0) node[left] {} circle (5pt) ; 
    \draw [fill=white,line width=0.6pt] (1.73,0) node[right] {} circle (5pt) ; 
    \draw [fill=white,line width=0.6pt] (0,1.14) node[below=2mm] {} circle (5pt); 
\end{tikzpicture}
        \caption{$W_1$}
        \label{fig:E1}
    \end{subfigure}
    \begin{subfigure}[t]{0.24\textwidth}
        \centering
\begin{tikzpicture}[scale=0.6]			
    \draw [line width=0.6pt, black] (0,3) to (0,1.14); 
    \draw [bend left=16,line width=0.6pt, black] (0,3) to (0,1.14); 
    \draw [bend right=16, line width=0.6pt, black] (0,3) to (0,1.14); 
    
    \draw [bend right=8, line width=0.6pt, black] (0,3) to (-1.73,0); 
    \draw [bend left=8, line width=0.6pt, black] (0,3) to (-1.73,0); 
    
    \draw [bend right=8, line width=0.6pt, black] (0,3) to (1.73,0); 
    \draw [bend left=8, line width=0.6pt, black] (0,3) to (1.73,0); 

    \draw [bend right=10, line width=0.6pt, black] (0,1.14) to (1.73,0); 
    \draw [bend left=10, line width=0.6pt, black] (0,1.14) to (1.73,0); 

    \draw [bend right=10, line width=0.6pt, black] (0,1.14) to (-1.73,0); 
    \draw [bend left=10, line width=0.6pt, black] (0,1.14) to (-1.73,0); 

    \draw [line width=0.6pt, black] (-1.73,0) to (1.73,0); 

    \draw [fill=white,line width=0.6pt] (0,3) node[above] {} circle (5pt) ; 
    \draw [fill=white,line width=0.6pt] (-1.73,0) node[left] {} circle (5pt) ; 
    \draw [fill=white,line width=0.6pt] (1.73,0) node[right] {} circle (5pt) ; 
    \draw [fill=white,line width=0.6pt] (0,1.14) node[below=2mm] {} circle (5pt); 
\end{tikzpicture}
        \caption{$W_2$}
        \label{fig:E2}
    \end{subfigure}    
    \caption{The 4 small graphs $aK_2$, $T_{a,b,c}$, $W_1$, and $W_2$}
    \label{fig:small graph with v(G)=4}
\end{figure}
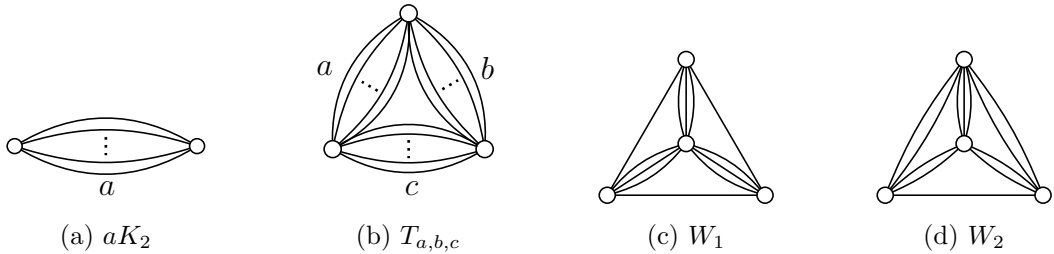

This paper investigates structural properties of graphs in $\SZ_5$. 
Intuitively, such graphs tend to arise under sufficiently high edge-connectivity and density. 
Rather than working directly with these parameters, we introduce a weight function inspired by that in~\cite{CL20}. 
Our Structural Theorem, together with this weight function, is also closely motivated by the Nash-Williams--Tutte Theorem~\cite{N61,T61}, which asserts that a graph $G$ contains $k$ edge-disjoint spanning trees if and only if every vertex partition $\{V_1, \dots, V_t\}$ of $G$ satisfies
\(
\sum_{i=1}^t d(V_i) - 2kt + 2k \geqslant 0.
\)

\begin{definition}\label{def:w(G)}
Let $G$ be a graph and $\mathcal{P}$ be a vertex partition of $G$. Denote $\mathcal{P}$ by $\{V_1, \dots, V_t\}$. The \emph{weight function} $w_G(\mathcal{P})$ of $\mathcal{P}$ is defined by
\[
w_G(\mathcal{P}) := \sum_{i=1}^{t} d(V_i) - 10t + 16.
\]
The \emph{weight} of $G$, denoted by $w(G)$, is the minimum of $w_G(\mathcal{P})$ over all partitions 
$\mathcal{P}$ of $G$.
\end{definition}

Let
\[
\mathcal{N}_5 := \{2K_2,\, 3K_2,\, T_{1,3,3},\, T_{2,2,3},\, W_1,\, W_2\}.
\]
The subscript ``$5$'' indicates that none of the graphs in $\mathcal{N}_5$ belongs to $\SZ_5$, as we will show later, in \Cref{thm:small-SZl-other}.
(Below $v(G):=|V(G)|$.)

\begin{definition}\label{def:S5-contractible}
Let $G$ be a connected graph with $v(G) \ge 2$. 
We say that $G$ is \emph{$\S_5$-contractible} if $w(G) \ge 0$ and $G/\mathcal{P} \notin \mathcal{N}_5$ for every vertex partition $\mathcal{P}$ of $G$.
\end{definition}

We can now state our recursive structural result. We defer its proof to \Cref{sec:proof of Thm}.

\begin{Stheorem}
If $G$ is an $\S_5$-contractible planar graph, then either: 
\begin{enumerate}[label=(\arabic*)]
\setlength{\itemsep}{0em}
\item\label{1-thm:F-w(G)-reduction} $v(G) \le 4$; or
\item\label{2-thm:F-w(G)-reduction} $G$ contains an $\S_5$-contractible proper subgraph; or
\item\label{3-thm:F-w(G)-reduction} by lifting certain paths in $G$, the resulting graph $G'$ is planar and contains an $\S_5$-contractible subgraph $H$; furthermore, contracting $H$ yields an $\S_5$-contractible graph.
\end{enumerate}
\end{Stheorem}

We make one remark concerning~\ref{3-thm:F-w(G)-reduction} in the Structural Theorem.
Note that the graph $G'/H$ is planar, since contracting any connected subgraph of a planar graph preserves planarity. The Structural Theorem yields the following result, which implies ~\Cref{thm:10-planar-SZ5}.

\begin{Mtheorem}
Every $\S_5$-contractible planar graph belongs to $\SZ_5$.
\end{Mtheorem}

We prove our Main Theorem via a straightforward induction, using the Structural Theorem; first we recall some properties of $\mathrm{S}_5$-contractibility.
And the base case $v(G)\le 4$ requires some preparation. So we prove our Main Theorem at the end of~\Cref{sec:proof-main-results}.

\section{Preliminaries}\label{sec:preliminary}

Let $G$ be a graph. Denote by $v(G)$ and $e(G)$ the numbers of vertices and edges of $G$. 
A \emph{$k$-path} is a path consisting of $k$ edges. 
Let $V_1 := \{v_1, \dots, v_m\}$, with $V_1\subseteq V(G)$.  We denote by $G[V_1]$ or 
by $G[v_1, \dots, v_m]$ the subgraph of $G$ induced by $V_1$. 
For brevity, if graphs $G_1$ and $G_2$ are isomorphic, then we simply write $G_1=G_2$.

For each $v \in V(G)$, let $N_G(v)$ denote the set of vertices adjacent to $v$. 
A \emph{$k$-vertex}, \emph{$k^+$-vertex}, or \emph{$k^-$-vertex}, is a vertex of degree exactly $k$, at least $k$, or at most $k$.
Denote by $\delta(G)$ the minimum degree of $G$.

For disjoint $X,Y \subseteq V(G)$, let $[X,Y]_G$ denote the set of edges with one endpoint in $X$ and the other in $Y$. When the graph $G$ is clear from context, the subscript $G$ is omitted. 
If $X=\{x\}$, we write $[x,Y]$ instead of $[\{x\},Y]$; similarly, if $Y=\{y\}$, we simply write $[x,y]$. 
For each $X \subseteq V(G)$, let $d(X) := |[X,X^c]|$; we call $d(X)$ the \emph{degree of $X$}.

For each pair $x,y \in V(G)$, let $\mu_G(x,y) := |[x,y]_G|$; this is the \emph{multiplicity} of $[x,y]$, i.e., the number of parallel edges between $x$ and $y$. 
The \emph{multiplicity} of $G$, denoted by $\mu(G)$, is given by $\mu(G) := \max\{\mu_G(x,y) : x,y \in V(G)\}.$

In this paper, we primarily study flow properties of planar graphs.
Let $G$ be a plane graph.
Denote by $F(G)$ the set of faces of $G$ and let $f(G) := |F(G)|$.
An edge or a vertex is \emph{incident} with a face $f$ if it lies on the boundary of $f$.
For a face $f \in F(G)$, let $V(f)$ denote the set of vertices incident with $f$; the 
\emph{degree} of $f$, denoted by $d(f)$, is the length of its boundary walk.
Analogous to vertex terminology, a \emph{$k$-face}, \emph{$k^+$-face}, or \emph{$k^-$-face} is a face of degree exactly $k$, at least $k$, or at most $k$.

Two faces $f$ and $f'$ are \emph{adjacent} if they share a common edge. 
A \emph{face chain} from $f_0$ to $f_t$ is a sequence of faces $f_0 \dots f_t$ with each consecutive pair sharing a common edge. 
Faces $f$ and $f'$ are \emph{weakly adjacent via $tK_2$} (equivalently, via $(t-1)$ $2$-faces) if there exists such a face chain $f_0 \dots f_t$ in which each intermediate face $f_i$ ($i\in\{1,\ldots,t-1\}$) is a 
$2$-face.  We allow $t=1$, in which case weak adjacency coincides with ordinary adjacency. 
When the specific value of $t$ is unimportant, we simply call $f$ and $f'$ \emph{weakly adjacent}.

\subsection{Two Recursive Methods: Edge Contraction and~\mbox{Path Lifting}}

Edge deletion is a widely used recursive technique in graph coloring problems. 
Its dual operation in a planar graph $G$ is \emph{edge contraction} in the planar dual $G^*$. 
Since the definition of edge contraction has already been given, we do not repeat it here. 
Contracting an edge in a planar graph preserves planarity and does not decrease edge connectivity. 
Consequently, this operation is quite useful in the study of flows and orientations.

In addition to edge contraction, we employ another recursive technique, \emph{path lifting}, which is well suited to inductive arguments. 
To \emph{lift} a path $v_0 v_1 \dots v_t$ in a graph $G$, we delete all of its edges  and add a new edge $v_0 v_t$. 
Although path lifting may reduce edge connectivity and may destroy planarity in general, it preserves planarity under certain conditions. 
In particular, if $v_0$ and $v_t$ lie on the boundary of a common face in a plane embedding of $G$, then lifting $v_0v_1\dots v_t$ results in a plane graph.

These two recursive methods—edge contraction and path lifting—serve as powerful tools for simplifying graph structures. 
Suppose that $G$ is a planar graph with a given $\Z_m$-boundary $\beta$. 
To establish the existence of a $\beta$-orientation of $G$, we may first lift appropriate paths in $G$ to simplify its structure. 
Next, we seek a subgraph $H$ with suitable properties in the resulting graph; such a subgraph is a \emph{reducible configuration}. 
If we find such an $H$, then we can contract $H$ and construct a $\beta$-orientation of the contracted graph, which we can then extend to a $\beta$-orientation of the original graph $G$. 
These steps will be used repeatedly in the subsequent proofs.

\subsection{Vertex Partition and Weight Function}

Throughout this paper, we do not consider the single-part partition of a graph $G$, unless $G=K_1$. A vertex partition $\mathcal{P}$ is \emph{trivial} if every part is a singleton, in which case $G/\mathcal{P}=G$. 
If each part of $\mathcal{P}$ induces a connected subgraph of $G$, then $G/\mathcal{P}$ can be viewed as the graph formed from $G$ by contracting the subgraph $G[V_i]$ for each $V_i\in\mathcal{P}$. 
In this situation, planarity is preserved: if $G$ is planar, then so is $G/\mathcal{P}$.

Let $\mathcal{P}$ be a partition of a graph $G$, denoted $\{V_1,\dots,V_t\}$, and let $n_i:=|V_i|$ whenever $i\in\{1,\ldots,t\}$. 
We define the \emph{type} of $\mathcal{P}$ as $(n_{i_1},\dots,n_{i_t})$, where $(i_1,\dots,i_t)$ is any 
permutation of $(1,\dots,t)$. 
To simplify notation, we write ``$*$'' to indicate omitted values, allowing us to emphasize only the most relevant part sizes. 
Thus, the type of $\mathcal{P}$ may be written as $(n_{i_1},\dots,n_{i_k},*)$ whenever $k\in\{1,\ldots,t\}$ and $n_{i_1},\dots,n_{i_k}$ are specified part sizes.

If a part size $n_{i_j}$ is known to be at least a given value $m_{i_j}$, then we write $m_{i_j}^+$ to indicate that $n_{i_j}\ge m_{i_j}$. 
This notation allows a more flexible description of types, for example,
\[
(m_{i_1}^+,m_{i_2}^+,\dots,m_{i_k}^+,n_{i_{k+1}},n_{i_{k+2}},*),
\]
where $m_{i_j}\le n_{i_j}$ for each $j\in\{1,\ldots,k\}$.

For instance, suppose $\mathcal{P}=\{V_1,\dots,V_t\}$, with $|V_i|=3$ for some $i$, $|V_j|=2$ for some $j\ne i$, and $|V_k|=1$ for all other $k$. 
Now the type of $\mathcal{P}$ can be written as $(3,2,1,\dots,1)$ or $(3,2,*)$ or $(2^+,2^+,*)$, depending on which part sizes we emphasize.
With this terminology, a partition $\mathcal{P}$ is \emph{trivial} if and only if it has type $(1,\dots,1)$; otherwise, it is \emph{nontrivial}.

Let $G$ be a graph and fix $H\subseteq G$ with $H$ connected. 
Contracting $H$ yields a new vertex $v_H$ in the graph $G/H$. 
Let $\mathcal{P}$ be a partition of $G/H$, denoted by $\{V_1,\dots,V_t\}$. 
If $v_H\in V_i$ for some $i\in\{1,\ldots,t\}$, then 
for all $j\ne i$ in $G$ we let $V_j':=V_j$, but we let
\[
V_i':=(V_i\setminus\{v_H\})\cup V(H).
\]
Let $\mathcal{P}_H:=\{V_1',\dots,V_t'\}$; note that $\mathcal{P}_H$ is a partition of $G$. We call $\mathcal{P}_H$ the \emph{$H$-restored partition} of $\mathcal{P}$. 
If $H$ is disconnected, then we apply this procedure to each of its components; the resulting partition is again the \emph{$H$-restored partition} of $\mathcal{P}$.

Recall that the \emph{weight function} of a partition $\mathcal{P}=\{V_1,\dots,V_t\}$ of $G$ is defined by
\[
w_G(\mathcal{P}) := \sum_{i=1}^{t} d(V_i) - 10t + 16.
\]
The \emph{weight} of $G$, denoted by $w(G)$, is the minimum value of $w_G(\mathcal{P})$ over all partitions $\mathcal{P}$ of $G$.
By definition, a larger value of $w(G)$ indicates a higher edge density in $G$. In particular, if $w(G)\ge 0$, then for the trivial partition $\mathcal{P}_0$ of $G$ we have that
\[
0 \le w(G) \le w_G(\mathcal{P}_0) = 2e(G) - 10v(G) + 16.
\]
This implies that
\begin{equation}
e(G) \ge 5v(G) - 8.
\label{edge-bound}
\tag{$\star$}
\end{equation}

If $G$ is $10$-edge-connected, then every partition $\mathcal{P} = \{V_1, \dots, V_t\}$ of $V(G)$,
satisfies
$
w_G(\mathcal{P}) \ge \sum_{i=1}^{t} 10 - 10t + 16 = 16 \ge 0,
$
and hence $w(G)\ge 0$.
Our structural results will be developed for graphs $G$ satisfying $w(G) \ge 0$, with a few exceptions.
To better understand the weight function, we have the following observations.

\begin{enumerate}[label=(\arabic*)]
\setlength{\itemsep}{0em}

\item
For every graph $G$ and partition $\mathcal{P}$, the value $w_G(\mathcal{P})$ is always even.
Let $G / \mathcal{P}$ be the graph formed from $G$ by contracting each part of $\mathcal{P}$.
If we let $G':=G/\mathcal{P}$, then
\[
w_G(\mathcal{P}) = w_{G'}(\mathcal{P}_0)
= 2e(G') - 10v(G') + 16 \equiv 0 \pmod{2},
\]
where $\mathcal{P}_0$ denotes the trivial partition of $G'$.

\item \label{connected}
Recall that $w(G)=\min_{\mathcal{P}} \{w_G(\mathcal{P})\}$. To compute $w(G)$, it suffices to 
consider partitions in which each part induces a connected subgraph of $G$.
Indeed, suppose that $\mathcal{P}=\{V_1,V_2,\dots,V_t\}$ and that $V_1=V_{11}\cup V_{12}$, where there are no edges between $V_{11}$ and $V_{12}$ in $G$.
Define a refined partition by letting
$
\mathcal{P}':=\{V_{11},V_{12},V_2,\dots,V_t\}.
$
A direct calculation shows that
\(
w_G(\mathcal{P}') = w_G(\mathcal{P}) - 10.
\) Moreover, if a graph $G$ satisfies $w(G)\ge0$ and is not $\S_5$-contractible, then there exists a partition $\P$ such that $G/\P\in\N_5$ where every part of $\mathcal{P}$ induces a connected subgraph of $G$. Otherwise, refining a part of $\mathcal{P}$ would result in a partition with weight at most $-8$, since every graph in $\N_5$ has weight at most $2$, contradicting $w(G)\ge 0$.

\item
For each graph $G$ and each subgraph $H$ of $G$, the contracted graph $G/H$ satisfies
\[
w(G/H) \ge w(G).
\]
Indeed, for any partition $\mathcal{P}$ of $G/H$, let $\mathcal{P}_H$ denote the corresponding $H$-restored partition of $G$.
So the inequality $w(G/H) \ge w(G)$ follows from the fact that

\[
w_{G/H}(\mathcal{P})= w_G(\mathcal{P}_H). 
\]
\end{enumerate}

In the proof of the Structural Theorem, given a graph $G$ with $w(G)\ge 0$, we often aim to find and 
contract a reducible configuration. This forms a smaller graph $G'$ with
\[
w(G') \ge w(G) \ge 0.
\]

But sometimes this approach fails. In such cases, we may first lift certain paths in $G$ to produce a modified graph in which a reducible configuration can be found and contracted. Nevertheless, there remain a few exceptional graphs for which this strategy fails and which exhibit undesirable structural properties.
In our Structural Theorem, $\mathcal{N}_5$ denotes a collection of small graphs that have non-negative weight but are too sparse to possess the desired structural properties. We make the following observation.

\begin{observation}\label{ob:weight-K2-Tabc}
The following $3$ statements all hold.
\begin{enumerate}[label=(\arabic*)]
\setlength{\itemsep}{0em}
    \item $w(aK_2) = 2a - 4$ for all $a \ge 1$.
    \item $w(T_{a,b,c}) = 2(a + b + c) - 14$ for all $0 \le a,b,c \le 5$.
    \item\label{w-N5}
    For graphs in $\mathcal{N}_5$, we have
    \[
	    w(3K_2) = 2\mbox{\quad and }
    \quad
    w(2K_2) = w(T_{1,3,3}) = w(T_{2,2,3}) = w(W_1) = w(W_2) = 0.
    \]
\end{enumerate}
\end{observation}

\begin{proof}
The graph $aK_2$, with $a \ge 1$, has only one possible partition type, $(1,1)$; hence
\[
w(aK_2) = 2a - 4.
\]
In particular, $w(2K_2) = 0$ and $w(3K_2) = 2$.

For $T_{a,b,c}$ with $0 \le a,b,c \le 5$, let $\mathcal{P}$ be a partition of it. Then
\[
w_{T_{a,b,c}}(\mathcal{P}) =
\begin{cases}
2(a+b+c)-14, & \text{if $\mathcal{P}$ has type $(1,1,1)$},\\[4pt]
2(x+y)-4\text{~where~}\{x,y\}\in\{\{a,b\},\{a,c\},\{b,c\}\},
& \text{if $\mathcal{P}$ has type $(2,1)$}.
\end{cases}
\]
Since $a,b,c \le 5$, the minimum value is attained when $\mathcal{P}$ has type $(1,1,1)$.
Therefore,
\[
w(T_{a,b,c}) = 2(a+b+c)-14.
\]
In particular, $w(T_{1,3,3}) = w(T_{2,2,3}) = 0$.

For the graphs $W_1$ and $W_2$, a direct calculation shows that the trivial partition yields the minimum weight. Indeed,
\(
w(W_1) = w(W_2) = 2 \times 12 - 10 \times 4 + 16 = 0.
\)
\end{proof}

\subsection{\texorpdfstring{$\S_5$}{S5}-Contractible Graphs}
In the inductive proof of the Structural Theorem, we observe that the graphs under consideration never reduce to a member of $\mathcal{N}_5$ through any sequence of edge contractions. This observation motivates the introduction of a class of graphs that are stable under contraction and exclude these exceptional configurations. The properties listed below describe the class of graphs that arise naturally in our inductive arguments and exclude the exceptional configurations in $\mathcal{N}_5$.

\begin{proposition}\label{prop:property-S5}
Let $G$ be a connected graph with $v(G) \ge 2$.
\begin{enumerate}[label=(\arabic*)]
\setlength{\itemsep}{0em}
    \item\label{1_prop:property-S5}
    If $G$ is not $\S_5$-contractible, then $w(G) \le 2$.
    
    \item\label{2_prop:property-S5}
    If $G$ is $\S_5$-contractible, then $G$ is $4$-edge-connected and $e(G) \ge 5v(G) - 8$.
    
    \item\label{3_prop:property-S5}
    If $G$ is $\S_5$-contractible, then every graph $G'$ formed by contracting a subset of edges of $G$ is also $\S_5$-contractible, and moreover $w(G') \ge w(G)$.
    
    \item\label{4_prop:property-S5}
    If $G$ is $\S_5$-contractible, then every graph $G'$ formed from $G$ by adding edges (without adding new vertices) is also $\S_5$-contractible. Moreover, $w(G') \ge w(G)$.
\end{enumerate}
\end{proposition}

\begin{proof}
Suppose first that $G$ is not $\S_5$-contractible.
Then either $w(G) < 0$, or there exists a partition $\mathcal{P}$ of $G$ such that $G/\mathcal{P} \in \mathcal{N}_5$.
In either case, we have $w(G) \le 2$, since $w(G_0) \le 2$ for every graph $G_0 \in \mathcal{N}_5$ by \Cref{ob:weight-K2-Tabc}.
This proves~\ref{1_prop:property-S5}.

Now assume that $G$ is $\S_5$-contractible.
Suppose there exists an edge-cut $[X, X^c]$ with $d(X) \le 3$, and consider the partition
\(
\mathcal{P} = \{X, X^c\}.
\)
Then the contracted graph $G/\mathcal{P}$ is $3K_2$, which belongs to $\mathcal{N}_5$, contradicting the $\S_5$-contractibility of $G$.
Therefore, $G$ is $4$-edge-connected.
The inequality $e(G) \ge 5v(G) - 8$ 
is inequality \eqref{edge-bound} above.
This completes the proof of~\ref{2_prop:property-S5}.

It remains to prove~\ref{3_prop:property-S5} and~\ref{4_prop:property-S5}.
It suffices to consider the cases in which a single edge is contracted or added, as the general statements then follow by induction.

First, let $G':= G / e$ for some $e = xy \in E(G)$.  
Since $G$ is $\S_5$-contractible, $w(G') \ge w(G) \ge 0$.  
Let $H := G[\{x,y\}]$. If there exists a partition $\mathcal{P}$ of $G'$ with $G'/\mathcal{P} \in \mathcal{N}_5$, then the corresponding $H$-restored partition $\mathcal{P}_H$ yields $G/\mathcal{P}_H \in \mathcal{N}_5$, a contradiction.

Next, let $G' := G + e$ for some $e = xy$ with $x,y \in V(G)$.  
Again, $w(G') \ge w(G) \ge 0$.  
If there exists a partition $\mathcal{P}$ of $G'$ with $G'/\mathcal{P} \in \mathcal{N}_5$, then either $G/\mathcal{P} \in \mathcal{N}_5$, contradicting the $\S_5$-contractibility of $G$, or $G/\mathcal{P}$ is obtained from some $G_0 \in \mathcal{N}_5$ by deleting a single edge.  
In the latter case, if $G_0 = 3K_2$, then $G/\mathcal{P} = 2K_2 \in \mathcal{N}_5$, again a contradiction; otherwise, $w(G/\mathcal{P}) = -2 < 0$, implying $w(G) < 0$, also a contradiction.

This completes the proof of~\ref{3_prop:property-S5} and~\ref{4_prop:property-S5}.
\end{proof}

Identifying $\S_5$-contractible graphs with small order is essential. We now characterize all $\S_5$-contractible graphs with order at most $4$.

\begin{theorem}\label{thm:small-S5-contractible}
Let $G$ be a graph with $v(G) \le 4$. Then $G$ is $\S_5$-contractible if and only if one of the following holds.
\begin{enumerate}[label=(\arabic*)]
\setlength{\itemsep}{0em}
    \item\label{1_thm:small-S5-contractible} $v(G) = 2$ and $e(G) \ge 4$; equivalently, $G=aK_2$ with $a \ge 4$.  
    \item\label{2_thm:small-S5-contractible} $v(G) = 3$, $e(G) \ge 8$, and $\delta(G) \ge 4$; equivalently, $G =T_{a,b,c}$ with $a+b+c \ge 8$ and $\min\{a+b,a+c,b+c\} \ge 4$.  
    \item\label{3_thm:small-S5-contractible} $v(G) = 4$ and $G$ has a spanning subgraph $G'$ with $e(G') \ge 12$, $\delta(G') \ge 4$, $\mu(G') \le 4$, and $G' \notin \{W_1, W_2\}$.  
\end{enumerate}
\end{theorem}

\begin{proof}
We proceed by cases on $v(G)$.

\textbf{Case 1: $\bm{v(G) = 2}$.}  
If $G$ is $\S_5$-contractible, then \Cref{prop:property-S5}\ref{2_prop:property-S5} implies $e(G) \ge 4$.  
Conversely, if $e(G) \ge 4$, then 
\(
w(G) = 2e(G) - 10\times 2 + 16 \ge 0.
\)  
Since the only possible partition is trivial and $G/\mathcal{P} = aK_2 \notin \mathcal{N}_5$ for $a \ge 4$, it follows that $G$ is $\S_5$-contractible.

\textbf{Case 2: $\bm{v(G) = 3}$.}  
If $G$ is $\S_5$-contractible, then \Cref{prop:property-S5}\ref{2_prop:property-S5} implies $\delta(G) \ge 4$ and $e(G)\ge7$.
But if $e(G) = 7$ and $\delta(G)\ge 4$, then $G\in\{T_{1,3,3},T_{2,2,3}\}\subseteq\mathcal{N}_5$, contradicting $\S_5$-contractibility. Hence, $e(G) \ge 8$.

Conversely, if $G=T_{a,b,c}$ with $a+b+c \ge 8$ and $\min\{a+b,a+c,b+c\} \ge 4$, then for each partition $\mathcal{P}$ of $G$, the graph $G/\mathcal{P}$ is either $G$ itself or a $2$-vertex graph with at least $4$ parallel edges.  
Thus, $G/\mathcal{P} \notin \mathcal{N}_5$.
Finally, $G$ is $\S_5$-contractible since
\[
w(G) \ge \min\big\{2\times 8 - 30 + 16,\ 2\times 4 - 20 + 16\big\} \ge 0.
\]

\textbf{Case 3: $\bm{v(G) = 4}$.}  
If $G$ is $\S_5$-contractible, then \Cref{prop:property-S5}\ref{2_prop:property-S5} implies $\delta(G) \ge 4$ and $e(G)\ge12$.
If $\mu(G) \le 4$, then let $G' := G$, and we are done.   
Otherwise, assume $\mu(G) \ge 5$ for some edge $x_1x_2$. Let $G_1 := G / G[x_1,x_2]$ and denote the contracted vertex by $x$. By \Cref{prop:property-S5}\ref{3_prop:property-S5}, $G_1$ is $\S_5$-contractible with 
$3$ vertices $x,y,z$.  
If any edge in $G_1$ has multiplicity at least $5$, then we can repeatedly remove excess edges while 
maintaining $e \ge 8$, $\delta \ge 4$; 
eventually we obtain a spanning subgraph 
$G_1' \subseteq G_1$ with these properties and also with $\mu\le 4$.  
Undoing the contraction of $x_1x_2$ in $G_1'$ yields a spanning subgraph $G' \subseteq G$, where the edge 
$x_1x_2$ is restored with multiplicity $4$ (all other edges unchanged). Then $G'$ satisfies $e(G') \ge 12$, 
$\delta(G') \ge 4$, $\mu(G') \le 4$, and $G' \notin\{W_1, W_2\}$.

Conversely, if $G$ contains such a spanning subgraph $G'$, then for each partition $\mathcal{P}$ of $G'$, we can check that $G'/\mathcal{P} \notin \mathcal{N}_5$ and $w_{G'}(\mathcal{P}) \ge 0$.  
Indeed, trivial, $2$-part, and $3$-part partitions $\mathcal{P}$ each satisfy $w_{G'}(\mathcal{P}) \ge 0$ by the bounds on $e(G')$, $\delta(G')$, and $\mu(G')$.  
Hence $G'$ is $\S_5$-contractible, and by \Cref{prop:property-S5}\ref{4_prop:property-S5}, so is $G$.
\end{proof}

\section[ProofOfS5]{Proof of the Main Theorem}\label{sec:proof-main-results}

Building on the Structural Theorem stated above, we now apply it to prove \Cref{thm:10-planar-SZ5}.
We begin by summarizing several basic properties of all graphs in $\SZ_\ell$; we also characterize all such 
graphs with small order.

\begin{proposition}\label{prop:SZ5-property}
Fix a graph $G$ and $H\subseteq G$.
If $\ell$ is an odd integer and $\ell \ge 3$, then
the following statements hold:
\begin{enumerate}[label=(\arabic*)]
\setlength{\itemsep}{0em}
    \item $K_1 \in \SZZ{\ell}$.
    \item\emph{\cite{L08}}
    If $G \in \SZZ{\ell}$, then both $G/e$ and $G+e'$ belong to $\SZZ{\ell}$ for each edge $e \in E(G)$ and each new edge $e'$ with endpoints in $V(G)$ but $e' \notin E(G)$.
    \item\label{thm:SZl-contraction-reserve}\emph{\cite{L08}}
    If $H \in \SZZ{\ell}$ and $G/H \in \SZZ{\ell}$, then $G \in \SZZ{\ell}$.
    \item\label{thm:SZl-lifting}
    If $G'$ is formed from $G$ by lifting certain paths and $G' \in \SZZ{\ell}$, then $G \in \SZZ{\ell}$.
\end{enumerate}
\end{proposition}

Statement~(1) is immediate.
Statements~(2) and (3) are well known; see \cite{L08} for detailed proofs.
For~(4), observe that any $\Z_{\ell}$-boundary $\beta$ of $G$ is also a $\Z_\ell$-boundary of $G'$, and the corresponding $\beta$-orientation of $G'$ can be naturally extended to a $\beta$-orientation of $G$. Therefore, $G \in \SZZ{\ell}$.

\medskip
For graphs $G$ with $v(G)\le 4$, the following results were proved in \cite{LSWW24}.

\begin{theorem}[\cite{LSWW24}]\label{thm:small-SZl-other}
Let $\ell$ be an odd integer with $\ell \ge 3$.
\begin{enumerate}[label=(\arabic*)]
\setlength{\itemsep}{0em}
\item\label{thm:SZl-small-2vertices}
The graph $aK_2$ is in $\SZ_\ell$ if and only if it contains $\ell-1$ edge-disjoint spanning trees, that is, if and only if $a \ge \ell-1$.
\item\label{thm:SZl-small-3vertices}
The graph $T_{a,b,c}$ is in $\SZ_\ell$ if and only if it contains $\ell-1$ edge-disjoint spanning trees, that is,
$a+b+c \ge 2\ell-2$ and $\min\{a+b,a+c,b+c\} \ge \ell-1$.
\item\label{thm:SZl-small-4vertices}
Let $G$ be a graph with $v(G)=4$. If $e(G)\ge 3\ell-2$, $\mu(G)\le \ell-2$, and $\delta(G)\ge \ell-1$,
then $G \in \SZZ{\ell}$.
\end{enumerate}
\end{theorem}

We now prove that every planar $\S_5$-contractible graph is in $\SZ_5$ (by assuming the Structural Theorem).
\medskip
\newpage

\noindent
{\bf Main Theorem.} 
\emph{Every planar $\S_5$-contractible graph belongs to $\SZ_5$.}

\begin{proof}
We proceed by induction on $v(G)$.

If $v(G) \in \{2,3\}$, then $G$ is $\S_5$-contractible if and only if $G \in \SZ_5$, 
by the characterizations in \Cref{thm:small-S5-contractible} and \Cref{thm:small-SZl-other}.

If $v(G) = 4$, then computer-aided verification 
(see \href{https://github.com/SuBoll/Graph-Flow-Experiments/tree/main/odd-SZl-4v-identifier}{https://github.com/SuBoll/Graph-Flow-Experiments/tree/main/odd-SZl-4v-identifier}) 
shows that any graph $G$ on $4$ vertices with at least $12$ edges and containing $4$ edge-disjoint spanning trees 
belongs to $\SZ_5$, unless $G \in \{W_1, W_2\}$. By \Cref{thm:small-S5-contractible}, a $4$-vertex graph $G$ is $\S_5$-contractible 
if and only if it contains a spanning subgraph $G'$ satisfying $e(G') \ge 12$, 
$\delta(G') \ge 4$, $\mu(G') \le 4$, and $G' \notin \{W_1, W_2\}$. 
Therefore, each $4$-vertex graph is $\S_5$-contractible if and only if it belongs to $\SZ_5$. 
Hence, all $\S_5$-contractible graphs with at most $4$ vertices are contained in $\SZ_5$.

Now suppose $G$ is $\S_5$-contractible and planar, with $v(G)\ge 5$. By the Structural Theorem, either 
\begin{enumerate}[label=(\alph*)]
\setlength{\itemsep}{0em}
    \item $G$ has an $\S_5$-contractible proper subgraph $H$ where $G/H$ is also $\S_5$-contractible; or
    \item after lifting certain paths of $G$, the resulting graph contains an $\S_5$-contractible subgraph $H$, and contracting $H$ yields an $\S_5$-contractible planar graph $G'$.
\end{enumerate}
In case~(a), the induction hypothesis guarantees $H,G/H\in\SZ_5$. 
In case~(b), now $H,G'\in \SZ_5$. So in both cases it follows from \Cref{prop:SZ5-property}\ref{thm:SZl-contraction-reserve} and~\ref{thm:SZl-lifting} that $G \in \SZ_5$.

Thus, if $G$ is planar and $G$ is $\S_5$-contractible, then $G\in \SZ_5$.
\end{proof}

We remark that the $4$-vertex case can also be verified by hand without computer. Since the verification is lengthy and highly technical, a full characterization of $4$-vertex graphs in $\SZ_\ell$ is provided separately in a note (\href{https://arxiv.org/abs/2603.21591}{arXiv 2603.21591}), to which we refer interested readers for more details.

We now give a proof of \Cref{thm:5treesSZ5}.

\begin{proof}[Proof of \Cref*{thm:5treesSZ5}]
By the Nash-Williams--Tutte Theorem~\cite{N61,T61}, if a graph $G$ contains $5$ edge-disjoint spanning trees, then $\sum_{i=1}^t d(V_i) - 10t + 10 \ge 0$ for every vertex partition $\{V_1,\ldots,V_t\}$ of $G$, which implies that $\sum_{i=1}^t d(V_i) - 10t + 16 \ge 6$. Hence, $w(G)\ge 6\geq 0$. Since $w(H)\le 2$ for every graph $H\in \N_5$, we have $G/\P\notin \N_5$ for every partition $\P$ of $G$. Thus, by~\Cref{def:S5-contractible}, $G$ is $\S_5$-contractible, and the~\textbf{Main Theorem} yields $G\in S\Z_5$.
\end{proof}

Recall that every $2k$-edge-connected graph contains $k$ edge-disjoint spanning trees. Therefore, \Cref{thm:5treesSZ5} immediately implies the following result.

\medskip
\noindent
{\bf \Cref*{thm:10-planar-SZ5}}
\emph{Every $10$-edge-connected planar graph is strongly $\Z_5$-connected.}

\section{Proof of the Structural Theorem}\label{sec:proof of Thm}

For two partitions $\mathcal{P} = \{V_1, \dots, V_t\}$ and $\mathcal{P}' = \{V_1', \dots, V_s'\}$ of a graph $G$, we call $\mathcal{P}'$ a \emph{refinement} of $\mathcal{P}$, or say that $\mathcal{P}'$ \emph{refines} $\mathcal{P}$, if $s>t$ and for each $V_j' \in \mathcal{P}'$ there exists $V_i \in \mathcal{P}$ with $V_j' \subseteq V_i$. Informally, $\mathcal{P}'$ is formed by further subdividing some parts of $\mathcal{P}$. The following lemma will be instrumental in proving our gap lemmas.

\begin{lemma}[Refinement Lemma]\label{lemma:refine-2}
Let $G$ be a graph and let $\mathcal{P} := \{V_1, \dots, V_t\}$, where $\mathcal{P}$ is a partition of $G$. 
For each $i \in \{1,\dots,t\}$, let $H_i := G[V_i]$ and let $\mathcal{Q}_i$ be a partition of $H_i$. 
For each $\ell \in\{1,\ldots,t\}$, let
\[
\mathcal{P}_\ell := \mathcal{Q}_1 \cup \dots \cup \mathcal{Q}_\ell \cup (\mathcal{P} \setminus \{V_1, \dots, V_\ell\}).
\]
Then
\[
w_G(\mathcal{P}) = \sum_{i=1}^\ell \bigl(6 - w_{H_i}(\mathcal{Q}_i)\bigr) + w_G(\mathcal{P}_\ell).
\]
\end{lemma}

\begin{proof}
	For each $k \in \{1,\dots,t\}$, we denote $\mathcal{Q}_k$ by $\{W_1, \dots, W_{s_k}\}$. Now we have
\begin{align*}
w_G(\mathcal{Q}_k \cup (\mathcal{P} \setminus \{V_k\})) 
	&= \Bigl[\sum_{i=1}^{s_k} d_G(W_i) + \sum_{j=1}^t d_G(V_j) - d_G(V_k)\Bigr] - 10(s_k + t - 1) + 16 \\
	&= \Bigl[\sum_{i=1}^{s_k} d_G(W_i) - d_G(V_k) - 10s_k + 16\Bigr] + \Bigl[\sum_{j=1}^t d_G(V_j) - 10t + 16\Bigr] - 6 \\
&= w_{H_k}(\mathcal{Q}_k) + w_G(\mathcal{P}) - 6.
\end{align*}
This gives
\begin{align}
w_G(\mathcal{P}) = (6 - w_{H_k}(\mathcal{Q}_k)) + w_G(\mathcal{Q}_k \cup (\mathcal{P} \setminus \{V_k\})).
\label{key-clm}
\end{align}

We now prove the lemma by induction on $\ell$. The base case $\ell = 1$ is immediate:
\[
w_G(\mathcal{P}) = (6 - w_{H_1}(\mathcal{Q}_1)) + w_G(\mathcal{P}_1).
\]

Assume instead that the statement holds for some $\ell<t$, that is,
\begin{align}
w_G(\mathcal{P}) = \sum_{i=1}^\ell (6 - w_{H_i}(\mathcal{Q}_i)) + w_G(\mathcal{P}_\ell).
\label{refinement-IH}
\end{align}
Observe that
$
\mathcal{P}_{\ell+1} = \mathcal{Q}_{\ell+1} \cup (\mathcal{P}_\ell \setminus \{V_{\ell+1}\}).
$
Applying \eqref{key-clm} to $V_{\ell+1}$ yields
\begin{align}
w_G(\mathcal{P}_\ell) = (6 - w_{H_{\ell+1}}(\mathcal{Q}_{\ell+1})) + w_G(\mathcal{P}_{\ell+1}).
\label{refinement-extra}
\end{align}
	To finish the inductive step, we simply substitute \eqref{refinement-extra} into \eqref{refinement-IH}:
\begin{align*}
w_G(\mathcal{P}) = \sum_{i=1}^{\ell+1} (6 - w_{H_i}(\mathcal{Q}_i)) + w_G(\mathcal{P}_{\ell+1}).
\end{align*}
\aftermath
\end{proof}

We define the \emph{co-weight} of a graph $H$ to be $\sigma(H) := 6 - w(H)$. In the lemma, if each $w(H_i)$ is attained by $\mathcal{Q}_i$, i.e., $w(H_i) = w_{H_i}(\mathcal{Q}_i)$, then for all $\ell \in \{1,\dots,t\}$,
\[
w_G(\mathcal{P}) = \sum_{i=1}^\ell \sigma(H_i) + w_G(\mathcal{P}_\ell).
\]

In the rest of the paper, let $G$ be a counterexample to the Structural Theorem minimizing $v(G) + e(G)$; for short, we say a ``minimum $G$.'' Note the following. 

\begin{lemma}\label{lem:property}
Our minimum $G$ is an $\S_5$-contractible planar graph such that:
\begin{enumerate}[label=(\arabic*)]
\setlength{\itemsep}{0em}
    \item $v(G) \ge 5$;\label{lem:property-vertices} and
    \item $G$ contains no $\S_5$-contractible proper subgraph;\label{lem:property-subgraph} and
    \item $G$ admits no path-lifting that produces a planar graph $G'$ with an $\S_5$-contractible subgraph $H$ such that $G'/H$ is also $\S_5$-contractible\label{lem:property-lifting}.
\end{enumerate}
\end{lemma}

Our proof strategy is to establish gap lemmas (Lemmas~\ref{lemma:partition-value} and \ref{lemma:gap-lemma}) 
and identify forbidden configurations of $G$. This ultimately leads to a contradiction via discharging.

\subsection{Basic Properties and Gap Lemmas}

Recall that $G$ contains no $\S_5$-contractible proper subgraph by~\Cref{lem:property}\ref{lem:property-subgraph}. So the following lemma is an immediate consequence of \Cref{prop:property-S5}\ref{1_prop:property-S5}. 

\begin{lemma}\label{lemma:subgraph_H_weight2}
Every proper subgraph $H$ of $G$, other than $K_1$, satisfies $w(H) \le 2$.
\end{lemma}

\begin{lemma}\label{lemma:>=4+weight}
Let $G'$ be a subgraph of $G$, and let $\P := \{V_1, \dots, V_t\}$, such that $\mathcal{P}$ is a nontrivial 
partition of $G'$ with $|V_k| \ge 2$ for some $k$. If $\Q$ is a partition of $V_k$ attaining the weight of 
$G'[V_k]$, i.e., $w_{G'[V_k]}(\Q) = w(G'[V_k])$, then
\[
w_{G'}(\P) \ge w_{G'}(\Q \cup (\P \setminus \{V_k\})) + 4.
\]
\end{lemma}

\begin{proof}
Let $H := G'[V_k]$. Since $H$ is a proper subgraph of $G$, by~\Cref{lemma:subgraph_H_weight2} we have $w_H(\Q) = w(H) \le 2$.  
By applying~\Cref{lemma:refine-2}, we get
\[
w_{G'}(\P) = w_{G'}(\Q \cup (\P \setminus \{V_k\})) + (6 - w(H)) \ge w_{G'}(\Q \cup (\P \setminus \{V_k\})) + 4.
\]
\aftermath
\end{proof}

\Cref{lemma:>=4+weight} shows that any nontrivial partition of a subgraph of $G$ can be refined to reduce its weight by at least $4$. Thus, when we evaluate the weight of a subgraph, we need only to consider its trivial partition.

\begin{lemma}\label{lem:trivial}
Suppose $G'\subseteq G$. A partition $\P$ of $G'$ is trivial if and only if $w(G') = w_{G'}(\P)$.
\end{lemma}

\begin{proof}
If $v(G') \le 2$, the only possible partition is trivial. Otherwise, assume $v(G') \ge 3$ and that a partition $\P$ of $G'$ satisfies $w_{G'}(\P) = w(G')$ but is not trivial.  
By \Cref{lemma:>=4+weight}, there exists a refinement $\P'$ of $\P$ such that
\[
w_{G'}(\P') \le w_{G'}(\P) - 4 = w(G') - 4.
\]
This contradicts the definition of $w(G')$.
\end{proof}

\Cref{lemma:subgraph_H_weight2} shows that every proper subgraph of $G$ has weight at most $2$.
We further refine this observation by characterizing all proper subgraphs with weight even smaller.
Recall from~\Cref{ob:weight-K2-Tabc}\ref{w-N5} that, 
$w(3K_2)=2$ and $w(H)=0$ for all $H\in\mathcal{N}_5\setminus\{3K_2\}$.

\begin{lemma}\label{lemma:is-N5}
Every proper subgraph $H$ of $G$ with $H \notin \mathcal{N}_5 \cup \{K_1\}$ satisfies $w(H) \le -2$.
\end{lemma}

\begin{proof}
Suppose instead that some proper subgraph $H$ of $G$ with
$H \notin \mathcal{N}_5 \cup \{K_1\}$ has $w(H) \ge 0$.
Recall that $G$ has no $\S_5$-contractible proper subgraph by~\Cref{lem:property}\ref{lem:property-subgraph}. So there exists a partition $\Q$ of $H$ such that $H/\Q \in \mathcal{N}_5$.
By~\Cref{ob:weight-K2-Tabc}\ref{w-N5}, this implies that
\(
w_H(\Q) \le 2.
\)

If $\Q$ is nontrivial, then by~\Cref{lemma:>=4+weight} some refinement $\Q'$ of $\Q$ satisfies 
\(
w_H(\Q') \le w_H(\Q) - 4 \le -2.
\)
This contradicts our assumption that $w(H) \ge 0$.
Thus, $\Q$ must be the trivial partition.
But now $H = H/\Q \in \mathcal{N}_5$, contradicting our choice of $H$.
\end{proof}

Combining \Cref{lemma:subgraph_H_weight2} and \Cref{lemma:is-N5}, we obtain the following corollary.

\begin{corollary}\label{cor:w(G)=0}
If $H$ is a proper subgraph of $G$, then the following statements hold.
\begin{enumerate}[label=(\arabic*)]
    \setlength{\itemsep}{0em}
    \item $w(H)=6$ if and only if $H=K_1$.
    \item $w(H)=2$ if and only if $H=3K_2$.
    \item $w(H)=0$ if and only if $H\in \mathcal{N}_5\setminus \{3K_2\}$.
\end{enumerate}
Moreover, $w(G)=0$, and hence $e(G)=5v(G)-8$.
\end{corollary}

\begin{proof}
We only prove the moreover part.
Suppose instead that $w(G)\ge 2$.
Let $e$ be an edge of $G$, and let $G':=G-e$.
Then
\(
w(G') \ge w(G)-2 \ge 0.
\)
Since $G'$ is a connected proper subgraph of $G$, \Cref{lemma:is-N5} implies that
$G' \in \mathcal{N}_5 \cup \{K_1\}$.
In particular, this yields $v(G)=v(G') \le 4$,
which contradicts~\Cref{lem:property}\ref{lem:property-vertices}.
\end{proof}

We now state an initial gap lemma, which will be strengthened later in~\Cref{lemma:gap-lemma}.

\begin{lemma}[Initial Gap Lemma]
\label{lemma:partition-value} 
Let $H$ be a subgraph of $G$, and let $\P := \{V_1, \dots, V_t\}$, where $\mathcal{P}$ is a partition of 
$H$. For each $i\in\{1,\ldots,t\}$, let $H_i := H[V_i]$ and let $\sigma_i := 6 - w(H_i)$. 
The following $3$ statements hold.
\begin{enumerate}[label=(\arabic*)]
\setlength{\itemsep}{0em}
    \item\label{C1-trivial}
    $ w(H_i) = 2e(H_i) - 10|V_i| + 16 $.
    \item\label{C2-increasing}
    \[
    \sigma_i =
        \begin{cases}
            0, & \text{if } H_i = K_1;\\
            4, & \text{if } H_i = 3K_2;\\
            6, & \text{if } H_i \in \mathcal{N}_5 \setminus \{3K_2\}.
        \end{cases}
    \]
    Otherwise, $\sigma_i \ge 8$.
    \item\label{C3-sum}
    $ w_H(\P) = \sum\limits_{i=1}^t \sigma_i + w(H) $. In particular, if $ H = G $, then
    $ w_G(\P) = \sum\limits_{i=1}^t \sigma_i $.
\end{enumerate}
\end{lemma}

\begin{proof}
\noindent
(1) By \Cref{lem:trivial}, the trivial partition of $H_i$ attains the weight of $H_i$, and hence
\(
w(H_i)=w_{H_i}(\P_0)=2e(H_i)-10|V_i|+16.
\)

\noindent
(2)
The stated values of $\sigma_i$ in the first $3$ cases follow directly from
\Cref{cor:w(G)=0}.
In all remaining cases, \Cref{cor:w(G)=0} yields $w(H_i)\le -2$, and therefore
\(
\sigma_i = 6 - w(H_i) \ge 8.
\)

\noindent
(3)
For each $i$, let $\Q_i$ be a partition of $H_i$ satisfying $w_{H_i}(\Q_i)=w(H_i)$.
By \Cref{lem:trivial}, each $\Q_i$ is trivial, so $\Q:=\Q_1\cup\cdots\cup\Q_t$ is the trivial
partition of $H$ and $w_H(\Q)=w(H)$.
Applying \Cref{lemma:refine-2}, we obtain
\[
w_H(\P)=\sum_{i=1}^t \sigma_i + w(H).
\]
If $H=G$, then $w(H)=0$ by \Cref{cor:w(G)=0}, and hence
$w_G(\P)=w_H(\P)=\sum_{i=1}^t \sigma_i$.
\end{proof}

Before improving this gap lemma, we introduce several special forbidden subgraphs. 
But first we consider the case where our minimum $G$ satisfies $v(G)=5$. We show that $G$ cannot be obtained by adding a single vertex to either $W_1$ or $W_2$.

\begin{lemma}\label{lemma:9-5_vertex}
For each $v\in V(G)$, we have $G-v\notin\{W_1,W_2\}$.
\end{lemma}

\begin{proof}
Assume instead that there exists $v\in V(G)$ such that $G-v\in\{W_1,W_2\}$. Let $W:=V(G)\setminus\{v\}=\{w_1,w_2,w_3,w_4\}$ and let $G_1:=G[W]$. Since $v(G)=5$, by \Cref{cor:w(G)=0} we have $e(G)=17$. Since $e(W_1)=e(W_2)=12$, we have $d_G(v)=5$. 

We label the vertices as follows: if $G_1$ is isomorphic to $W_1$, then we use the labeling shown in \Cref{fig:proof-E1}; if $G_1$ is isomorphic to $W_2$, then we use the labeling shown in \Cref{fig:proof-E2}.

\begin{figure}[htbp!]
\centering
\begin{subfigure}{0.48\textwidth}
\centering
\begin{tikzpicture}[scale=0.6]
    \draw [line width=0.5pt] (0,3) to (-1.73,0);
    \draw [line width=0.5pt] (0,3) to (1.73,0);
    \draw [line width=0.5pt] (-1.73,0) to (1.73,0);

    \draw [line width=0.5pt] (0,1.14) to (1.73,0);
    \draw [bend left=14,line width=0.5pt] (0,1.14) to (1.73,0);
    \draw [bend right=14,line width=0.5pt] (0,1.14) to (1.73,0);

    \draw [line width=0.5pt] (0,1.14) to (-1.73,0);
    \draw [bend left=14,line width=0.5pt] (0,1.14) to (-1.73,0);
    \draw [bend right=14,line width=0.5pt] (0,1.14) to (-1.73,0);

    \draw [line width=0.5pt] (0,1.14) to (0,3);
    \draw [bend left=14,line width=0.5pt] (0,1.14) to (0,3);
    \draw [bend right=14,line width=0.5pt] (0,1.14) to (0,3);

    \draw [fill=white,line width=0.5pt] (0,3) node[above] {$w_3$} circle (4pt);
    \draw [fill=white,line width=0.5pt] (-1.73,0) node[left] {$w_1$} circle (4pt);
    \draw [fill=white,line width=0.5pt] (1.73,0) node[right] {$w_2$} circle (4pt);
    \draw [fill=white,line width=0.5pt] (0,1.14) node[below=2mm] {$w_4$} circle (4pt);
\end{tikzpicture}
\caption{$G_1= W_1$}
\label{fig:proof-E1}
\end{subfigure}
\hfill
\begin{subfigure}{0.48\textwidth}
\centering
\begin{tikzpicture}[scale=0.6]
    \draw [bend right=8,line width=0.5pt] (0,3) to (-1.73,0);
    \draw [bend left=8,line width=0.5pt] (0,3) to (-1.73,0);

    \draw [bend right=8,line width=0.5pt] (0,3) to (1.73,0);
    \draw [bend left=8,line width=0.5pt] (0,3) to (1.73,0);

    \draw [bend right=10,line width=0.5pt] (0,1.14) to (1.73,0);
    \draw [bend left=10,line width=0.5pt] (0,1.14) to (1.73,0);

    \draw [bend right=10,line width=0.5pt] (0,1.14) to (-1.73,0);
    \draw [bend left=10,line width=0.5pt] (0,1.14) to (-1.73,0);

    \draw [line width=0.5pt] (0,3) to (0,1.14);

    \draw [bend left=10,line width=0.5pt] (-1.73,0) to (1.73,0);
    \draw [bend right=10,line width=0.5pt] (-1.73,0) to (1.73,0);
    \draw [line width=0.5pt] (-1.73,0) to (1.73,0);

    \draw [fill=white,line width=0.5pt] (0,3) node[above] {$w_1$} circle (4pt);
    \draw [fill=white,line width=0.5pt] (-1.73,0) node[left] {$w_3$} circle (4pt);
    \draw [fill=white,line width=0.5pt] (1.73,0) node[right] {$w_4$} circle (4pt);
    \draw [fill=white,line width=0.5pt] (0,1.14) node[below=1mm] {$w_2$} circle (4pt);
\end{tikzpicture}
\caption{$G_1= W_2$}
\label{fig:proof-E2}
\end{subfigure}
	\caption{$G_1\in\{W_1, W_2\}$.}
\label{fig: W1W2proof}
\end{figure}
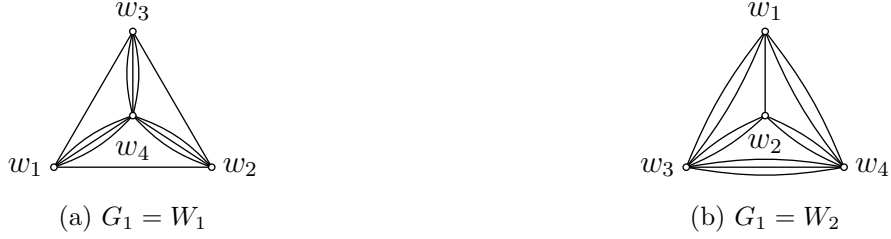

Note that $|N(v)\cap W|\le 3$; if not, then $G$ contains a copy of $K_5$, contradicting planarity. But since $4K_2$ is $\S_5$-contractible, we have $\mu(G)\le 3$. Thus, $|N(v)\cap W|\ge 2$. Next we show that in both cases (following the labels of \Cref{fig:proof-E1,fig:proof-E2}), without loss of generality, $d_G(w_2)\ge 6$ and $1\le \mu_G(v,w_2)\le 3$. 

Suppose first that $G_1= W_1$. As shown in \Cref{fig:proof-E1}, we have that $d_{G_1}(w_4)=9$ and $d_{G_1}(w_1)=d_{G_1}(w_2)=d_{G_1}(w_3)=5$. Since $v$ must have a neighbor in $\{w_1,w_2,w_3\}$, by symmetry we may assume that $vw_2\in E(G)$. Then $d_G(w_2)\ge 6$ and $1\le \mu_G(v,w_2)\le 3$.

Suppose instead that $G_1=W_2$. As shown in \Cref{fig:proof-E2}, we have that $d_{G_1}(w_1)=d_{G_1}(w_2)=5$ and $d_{G_1}(w_3)=d_{G_1}(w_4)=7$. If $v$ is adjacent only to $w_3$ and $w_4$, then the induced subgraph $G[\{v,w_3,w_4\}]$ is a multi-triangle with $8$ edges and minimum degree at least $4$, which is $\S_5$-contractible by \Cref{thm:small-S5-contractible}\ref{2_thm:small-S5-contractible}, a contradiction. Hence, $v$ has a neighbor
in $\{w_1,w_2\}$; by symmetry, we may assume that $vw_2\in E(G)$. Again, we have $d_G(w_2)\ge 6$ and $1\le \mu_G(v,w_2)\le 3$.

In either case, we lift the $2$-path $w_1w_2w_3$, and let $H$ be the subgraph induced by $\{w_1,w_3,w_4\}$. 
Note that $H=T_{2,3,3}$, so $H$ is $\S_5$-contractible. Starting from $G$, lifting the path $w_1w_2w_3$, and contracting $H$ yields a graph $G_0$ with $v(G_0)=3$ and $e(G_0)=8$. Let $w$ denote the vertex in $G_0$ formed from contracting $H$. Since $|[w_2,\{w_1,w_3,w_4\}]|_G=5$, we have $\mu_{G_0}(w,w_2)=3$. Moreover, $d_G(v)\ge5$, $1\le \mu_{G_0}(v,w_2)\le 3$, and $2\le \mu_{G_0}(v,w)\le 4$, which together imply that $\delta(G_0)\ge 4$. By \Cref{thm:small-S5-contractible}\ref{2_thm:small-S5-contractible}, $G_0$ is $\S_5$-contractible, contradicting \Cref{lem:property}\ref{lem:property-lifting}.
\end{proof}

\begin{lemma}\label{lem:6-edge-connected}
The graph $G$ is $6$-edge-connected.
\end{lemma}

\begin{proof}
Let $\P=\{X,X^c\}$ be a partition of $G$, where $|X|\ge |X^c|\ge 1$. Since $v(G)\ge 5$ by \Cref{lem:property}\ref{lem:property-vertices}, we have either $|X|\ge 3$ and $|X^c|\ge 2$, or $|X|\ge 4$ and $|X^c|=1$.

In the former case, the Initial Gap Lemma (\Cref{lemma:partition-value}) gives $\sigma(G[X])\ge 6$ and $\sigma(G[X^c])\ge 4$, and hence $w_G(\P)\ge 10$. In the latter case, $G[X^c]= K_1$. So \Cref{lemma:9-5_vertex} gives
that $G[X]\notin\{W_1,W_2\}$. Now again by \Cref{lemma:partition-value} we have $w_G(\P)\ge 8$.

Thus, for every bipartition $\P=\{X,X^c\}$ of $G$, we have $w_G(\P)\ge 8$. Since $w_G(\P)=2|[X,X^c]|-4$, we get $|[X,X^c]|\ge (8+4)/2 = 6$. That is, $G$ is $6$-edge-connected.
\end{proof}

\begin{lemma}\label{lemma:no-T113}
    The graph $G$ contains no $T_{1,1,3}$ as a subgraph.
\end{lemma}

\begin{figure}[htbp!]
    \centering
    \begin{subfigure}{0.28\textwidth}
           \centering
	    \begin{tikzpicture}[scale=1]
		    \draw [line width=0.5pt, black] (0,1.732) to (1,0);
		    \draw [line width=0.5pt, black] (0,1.732) to (-1,0);  
  		    \draw [line width=0.5pt, black] (1,0) to (-1,0);
  		    \draw [bend left=18,line width=0.5pt, black] (1,0) to (-1,0);
		    \draw [bend right=18,line width=0.5pt, black] (1,0) to (-1,0);
	    	\draw [fill=black,line width=0.2pt] (-1,0) node[left=0.5mm] {$v_2$} circle (2.5pt);
	    	\draw [fill=black,line width=0.2pt] (1,0) node[right=0.5mm] {$v_3$} circle (2.5pt);
	    	\draw [fill=white,line width=0.5pt] (0,1.732) node[above] {$v_1$} circle (2.5pt); 
	    \end{tikzpicture}
	    \caption{$T_{1,1,3}$.}
	    \label{fig:T113}
    \end{subfigure}
    \hfill
    \begin{subfigure}{0.35\textwidth}
    \centering
    \begin{tikzpicture}[scale=0.6]			
    \draw [bend right=6, line width=0.5pt, black] (0,2) to (-2.2,-1.3); 
    \draw [bend left=6, line width=0.5pt, black] (0,2) to (-2.2,-1.3); 
    \draw [bend right=6, line width=0.5pt, black] (0,2) to (2.2,-1.3); 
    \draw [bend left=6, line width=0.5pt, black] (0,2) to (2.2,-1.3); 
    \draw [bend right=8, line width=0.5pt, black] (-2.2,-1.3) to (2.2,-1.3); 
    \draw [bend left=8, line width=0.5pt, black] (-2.2,-1.3) to (2.2,-1.3); 
    \draw [line width=0.5pt, black] (-2.2,-1.3) to (2.2,-1.3); 
    \draw [bend right=8, line width=0.5pt, black] (0,2) to (-0.5,0.3); 
    \draw [bend left=8, line width=0.5pt, black] (0,2) to (-0.5,0.3); 
    \draw [line width=0.5pt, black] (0,2) to (0.5,0.3); 
    \draw [bend right=22, line width=0.5pt, black] (0.5,0.3) to (-0.5,0.3); 
    \draw [bend left=22, line width=0.5pt, black] (0.5,0.3) to (-0.5,0.3); 
    \draw [line width=0.5pt, black] (0.5,0.3) to (-0.5,0.3); 
    \draw [line width=0.5pt, black] (-2.2,-1.3) to (-0.5,0.3); 
    \draw [line width=0.5pt, black] (-2.2,-1.3) to (0.5,0.3); 
    \draw [bend right=7, line width=0.5pt, black] (2.2,-1.3) to (0.5,0.3); 
    \draw [bend left=7, line width=0.5pt, black] (2.2,-1.3) to (0.5,0.3); 
    \draw [fill=white,line width=0.5pt] (0,2) node[above] {$v_1$} circle (4pt) ; 
    \draw [fill=white,line width=0.5pt] (-0.5,0.3) node[below=0.7mm] {$v_2$} circle (4pt) ; 
    \draw [fill=white,line width=0.5pt] (0.5,0.3) node[below=0.7mm] {$v_3$} circle (4pt) ; 
    \draw [fill=white,line width=0.5pt] (-2.2,-1.3) node[left] {$v_4$} circle (4pt) ; 
    \draw [fill=white,line width=0.5pt] (2.2,-1.3) node[right] {$v_5$} circle (4pt) ; 
    \end{tikzpicture}
    \caption{A reconstruction of $G$.}
    \label{fig:E2-add-edge1}
    \end{subfigure}    
    \hfill
    \begin{subfigure}{0.35\textwidth}
    \centering
    \begin{tikzpicture}[scale=0.6]			
    \draw [bend right=6, line width=0.5pt, black] (0,2) to (-2.2,-1.3); 
    \draw [bend left=6, line width=0.5pt, black] (0,2) to (-2.2,-1.3); 
    \draw [bend right=6, line width=0.5pt, black] (0,2) to (2.2,-1.3); 
    \draw [bend left=6, line width=0.5pt, black] (0,2) to (2.2,-1.3); 
    \draw [bend right=8, line width=0.5pt, black] (-2.2,-1.3) to (2.2,-1.3); 
    \draw [bend left=8, line width=0.5pt, black] (-2.2,-1.3) to (2.2,-1.3); 
    \draw [line width=0.5pt, black] (-2.2,-1.3) to (2.2,-1.3); 
    \draw [bend right=8, line width=0.5pt, black] (0,2) to (-0.5,0.3); 
    \draw [bend left=8, line width=0.5pt, black] (0,2) to (-0.5,0.3); 
    \draw [line width=0.5pt, black] (0,2) to (0.5,0.3); 
    \draw [bend right=22, line width=0.5pt, black] (0.5,0.3) to (-0.5,0.3); 
    \draw [bend left=22, line width=0.5pt, black] (0.5,0.3) to (-0.5,0.3); 
    \draw [line width=0.5pt, black] (0.5,0.3) to (-0.5,0.3); 
    \draw [bend right=7, line width=0.5pt, black] (-2.2,-1.3) to (-0.5,0.3); 
    \draw [bend left=7, line width=0.5pt, black] (-2.2,-1.3) to (-0.5,0.3); 
    \draw [bend right=7, line width=0.5pt, black] (2.2,-1.3) to (0.5,0.3); 
    \draw [bend left=7, line width=0.5pt, black] (2.2,-1.3) to (0.5,0.3); 
    \draw [fill=white,line width=0.5pt] (0,2) node[above] {$v_1$} circle (4pt) ; 
    \draw [fill=white,line width=0.5pt] (-0.5,0.3) node[below=0.7mm] {$v_2$} circle (4pt) ; 
    \draw [fill=white,line width=0.5pt] (0.5,0.3) node[below=0.7mm] {$v_3$} circle (4pt) ; 
    \draw [fill=white,line width=0.5pt] (-2.2,-1.3) node[left] {$v_4$} circle (4pt) ; 
    \draw [fill=white,line width=0.5pt] (2.2,-1.3) node[right] {$v_5$} circle (4pt) ; 
\end{tikzpicture}
    \caption{Another reconstruction of $G$.}
    \label{fig:E2-add-edge2}
    \end{subfigure}
    \caption{The configuration $T_{1,1,3}$ and related reconstructions of $G$}
    \label{fig:NoT113}
\end{figure}
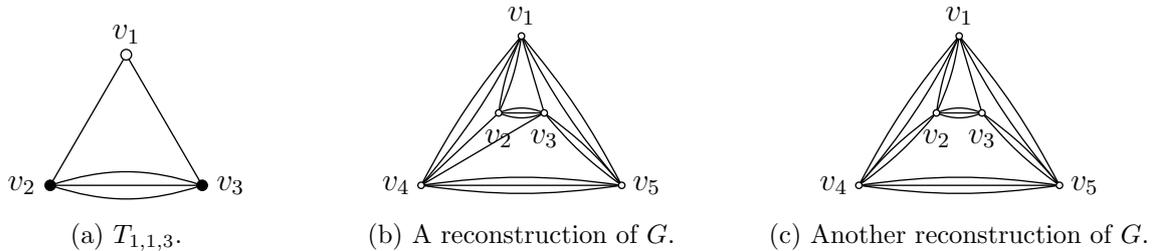

\begin{proof}
Suppose instead that $G$ contains a copy of $T_{1,1,3}$, as shown in \Cref{fig:T113}. We lift the path $v_2v_1v_3$ to form a graph $G_1$, let $H:=G_1[\{v_2,v_3\}]$, and note that $H=4K_2$. Form $G_2$ from $G_1$ by contracting $H$, and let $v_H$ denote the resulting new vertex. Now $V(G_2)=V(G)\cup\{v_H\}\setminus\{v_2,v_3\}$.

For each partition $\P$ of $G_2$, let $\P_H$ denote the corresponding $H$-restored partition of $G_1$. Note that $\P_H$ is also a partition of $G$ of type $(2^+,*)$. 
By the Initial Gap Lemma (\Cref{lemma:partition-value}), we have $w_G(\P_H)\ge 4$, so
$w_{G_2}(\P)=w_{G_1}(\P_H)\ge w_G(\P_H)-2\times 2\ge 0$.
Hence, $w(G_2)\ge 0$. Moreover, if $\P$ is nontrivial, then $\P_H$ has type $(3^+,*)$ or 
$(2^+,2^+,*)$, so $w_G(\P_H)\ge \min\{6,4+4\}=6$ and thus $w_{G_2}(\P)\ge 6-2\times2 = 2$.

By \Cref{lem:property}\ref{lem:property-lifting}, the graph $G_2$ is not $\S_5$-contractible. Hence, there exists a partition $\P'$ of $G_2$ such that $G_2/\P'\in\mathcal{N}_5$. Since $G$ is $6$-edge-connected by \Cref{lem:6-edge-connected}, the graph $G_2$ is $4$-edge-connected. Thus,
$G_2/\P'\in\{T_{1,3,3},T_{2,2,3},W_1,W_2\}$.
Every nontrivial partition $\mathcal{P}$ of $G_2$ has $w_{G_2}(\mathcal{P})\ge 2$. So $\P'$ must be trivial. 
Hence, $G_2=G_2/\P'\in\{T_{1,3,3},T_{2,2,3},W_1,W_2\}$.
Since $v(G_2)=v(G)-1\ge 4$, we have $G_2\in\{W_1,W_2\}$. 
Note that $d_{G_2}(v)=d_G(v)$ for all $v\in V(G_2)\setminus\{v_1,v_H\}$.
So by \Cref{lem:6-edge-connected}, each such $v$ satisfies $d_{G_2}(v)\ge 6$.

Note that $G_2\ne W_1$, as $W_1$ contains three $5$-vertices, contradicting that $G$ is $6$-edge-connected. 
Hence, $G_2=W_2$. In this case, $G_2$ has exactly two $5$-vertices, namely $v_1$ and $v_H$. Denote the remaining vertices of $G_2$ by $v_4$ and $v_5$, which are both $6^+$-vertices.
We now restore the graph $G$ from $G_2$. By symmetry, assume that $\mu_G(v_1,v_2)=2$, $\mu_G(v_1,v_3)=1$, and that $v_4$ is adjacent to $v_2$. 
Note that $v_4$ and $v_5$ cannot both be adjacent to each of $v_2$ and $v_3$, since otherwise $G$ would contain a $K_5$, contradicting the planarity of $G$. 
Moreover, if $v_5$ is adjacent to both $v_2$ and $v_3$, then $\mu_G(v_3,v_5)=1$ and $v_4$ cannot be adjacent to $v_3$. Consequently, $d_G(v_3)=5$, contradicting the $6$-edge-connectivity of $G$. Therefore, $G$ must be one of the two graphs shown in~\Cref{fig:E2-add-edge1,fig:E2-add-edge2}.

In each case, we lift the path $v_1v_2v_4$ and the new subgraph induced by $\{v_1,v_4,v_5\}$ is $T_{2,3,3}$, 
which is $\S_5$-contractible by \Cref{thm:small-S5-contractible}\ref{2_thm:small-S5-contractible}. Contracting this subgraph yields a copy of either $T_{1,3,4}$ (in \Cref{fig:E2-add-edge1}) or $T_{2,3,3}$ (in \Cref{fig:E2-add-edge2}), both of which are $\S_5$-contractible. This contradicts \Cref{lem:property}\ref{lem:property-lifting}.
\end{proof}

We can now strengthen our Initial Gap Lemma as follows.
The key improvement is that now fewer graphs $H_i$ can have $\sigma_i=6$.

\begin{lemma}[Gap Lemma]\label{lemma:gap-lemma}
Let $H$ be a subgraph of $G$, and let $\P=\{V_1,\dots,V_t\}$ be a partition of $H$. 
For each $i\in\{1,\ldots,t\}$, let $H_i:=H[V_i]$ and let $\sigma_i:=6-w(H_i)$.
The following $3$ statements hold.
\begin{enumerate}[label=(\arabic*)]
\setlength{\itemsep}{0em}
    \item 
    $w(H_i)=2e(H_i)-10|V_i|+16$.
    \item\label{C2-increasing-more}
    \[
    \sigma_i=
        \begin{cases}
            0, & \text{if } H_i= K_1,\\
            4, & \text{if } H_i= 3K_2,\\
            6, & \text{if } H_i= 2K_2.
        \end{cases}
    \]
    Otherwise, $\sigma_i\ge 8$.
    \item 
    $w_H(\P)=\sum_{i=1}^t \sigma_i + w(H)$. In particular, if $H=G$, then
    $w_G(\P)=\sum_{i=1}^t \sigma_i$.
\end{enumerate}
\end{lemma}

\begin{proof}
By \Cref{lemma:no-T113}, the graph $G$ contains no copy of $T_{1,1,3}$. 
But every graph in $\N_5\setminus\{2K_2,3K_2\}
=\{T_{1,3,3},T_{2,2,3},W_1,W_2\}$ contains a copy of $T_{1,1,3}$. Hence, none of these graphs appears as an induced subgraph of any $H_i$; this gives the stated values of $\sigma_i$.
\end{proof}

In the remainder of the proof, we frequently apply the Gap Lemma. Recall that the co-weight $\sigma(H)$ of a 
graph $H$ is defined by $\sigma(H) = 6 - w(H)$. By \Cref{lemma:gap-lemma}, if $v(H) \ge 2$, then $\sigma(H) \ge 4$; and if $v(H) \ge 3$, then $\sigma(H) \ge 8$.

A bipartition $\{X, X^c\}$ of a graph $H$ is an \emph{essential $k$-edge-cut} if $|X|, |X^c| \ge 2$ and $|[X,X^c]| = k$. And $H$ is \emph{essentially $k$-edge-connected} if each essential $\ell$-edge-cut has $\ell \ge k$.

\begin{lemma}\label{lem:essential-8-edge-connected}
    The graph $G$ is essentially $8$-edge-connected.
\end{lemma}

\begin{proof}
Let $[X,X^c]$ be an essential edge-cut in $G$ with $|X|\le |X^c|$, and let $\P := \{X, X^c\}$. Since $v(G)\ge 5$ by \Cref{lem:property}\ref{lem:property-vertices}, we have $|X|\ge 2$ and $|X^c|\ge 3$. By the Gap Lemma (\Cref{lemma:gap-lemma}),
    \(
        w_G(\P) = \sigma(G[X]) + \sigma(G[X^c]) \ge \min\{4,6\} + 8 \ge 12.
    \)
    But $w_G(\P) = 2|[X,X^c]| - 10 \times 2 + 16 = 2|[X,X^c]| - 4$, so $2|[X,X^c]| - 4 \ge 12$ implies $|[X,X^c]| \ge 8$. 
\end{proof}

The Gap Lemma (\Cref{lemma:gap-lemma}) provides a structural characterization of subgraphs $H \subseteq G$ with $w(H) \ge 0$. To handle the remaining cases, we further describe subgraphs $H \subsetneq G$ with $w(H) = -2$ and $w(H) = -4$. 

For each $k \in \{1,2\}$, let $\mathcal{F}_k$ denote the set of graphs $F'$ that can be formed from 
$\{2K_2, T_{1,3,3}, T_{2,2,3}, 2K_4\}$ by deleting $k$ edges, such that $F'$ contains 
no copy of $T_{1,1,3}$. Observe that \Cref{fig:4-graphs-w(G)=-2} shows $\mathcal{F}_1$. 
Note that each graph in $\mathcal{F}_2$ can be formed from some graph in $\mathcal{F}_1$ by deleting one additional edge. 
Moreover, each graph in $\mathcal{F}_k$ can have $k$ edges added, without creating an edge of multiplicity $5$, resulting in a graph that is not $\S_5$-contractible.

\begin{figure}[htbp!]
    \centering
    \begin{subfigure}{0.2\textwidth}
    \centering
        \begin{tikzpicture}[scale=0.8]			
    \draw [line width=0.5pt, black] (-1,1) to (1,1); 
    
    \draw [fill=white,line width=0.5pt] (-1,1) node[left] {} circle (3pt) ; 
    \draw [fill=white,line width=0.5pt] (1,1) node[right] {} circle (3pt) ; 
      \end{tikzpicture}
        \caption{$K_2$}
		\label{fig:1-w=-2-graph}
    \end{subfigure}
    \hfill
    \begin{subfigure}{0.2\textwidth}
    \centering
        \begin{tikzpicture}[scale=0.8]			
    \draw [bend left=15, line width=0.5pt, black] (-1,-0.58) to (0,1.15); 
    \draw [bend left=15, line width=0.5pt, black] (0,1.15) to (-1,-0.58); 
    \draw [line width=0.5pt, black] (0,1.15) to (-1,-0.58); 

    \draw [bend left=15, line width=0.5pt, black] (1,-0.58) to (0,1.15); 
    \draw [bend left=15, line width=0.5pt, black] (0,1.15) to (1,-0.58); 
    \draw [line width=0.5pt, black] (0,1.15) to (1,-0.58);

    \draw [fill=white,line width=0.5pt] (-1,-0.58) node[left] {} circle (3pt) ; 
    \draw [fill=white,line width=0.5pt] (1,-0.58) node[right] {} circle (3pt) ;
    \draw [fill=white,line width=0.5pt] (0,1.15) node[right] {} circle (3pt) ;
\end{tikzpicture}
        \caption{$2P_3$}
		\label{fig:2-w=-2-graph}
    \end{subfigure}
    \hfill
    \begin{subfigure}{0.2\textwidth}
    \centering
        \begin{tikzpicture}[scale=0.8]			
    \draw [bend left=15, line width=0.5pt, black] (-1,-0.58) to (0,1.15); 
    \draw [bend left=15, line width=0.5pt, black] (0,1.15) to (-1,-0.58); 

    \draw [bend left=15, line width=0.5pt, black] (1,-0.58) to (0,1.15); 
    \draw [bend left=15, line width=0.5pt, black] (0,1.15) to (1,-0.58); 

    \draw [bend left=15, line width=0.5pt, black] (-1,-0.58) to (1,-0.58); 
    \draw [bend right=15, line width=0.5pt, black] (-1,-0.58) to (1,-0.58); 
    
    \draw [fill=white,line width=0.5pt] (-1,-0.58) node[left] {} circle (3pt) ; 
    \draw [fill=white,line width=0.5pt] (1,-0.58) node[right] {} circle (3pt) ;
    \draw [fill=white,line width=0.5pt] (0,1.15) node[right] {} circle (3pt) ;
\end{tikzpicture}
        \caption{$2K_3$}
		\label{fig:3-w=-2-graph}
    \end{subfigure}
    \hfill
    \begin{subfigure}{0.2\textwidth}
    \centering
        \begin{tikzpicture}[scale=0.6]			
    \draw [bend left=10,line width=0.5pt, black] (0,3) to (0,1.14); 
    \draw [bend right=10, line width=0.5pt, black] (0,3) to (0,1.14); 
    
    \draw [bend right=8, line width=0.5pt, black] (0,3) to (-1.73,0); 
    \draw [bend left=8, line width=0.5pt, black] (0,3) to (-1.73,0); 
    
    \draw [bend right=8, line width=0.5pt, black] (0,3) to (1.73,0); 
    \draw [bend left=8, line width=0.5pt, black] (0,3) to (1.73,0); 

    \draw [bend right=10, line width=0.5pt, black] (0,1.14) to (1.73,0); 
    \draw [bend left=10, line width=0.5pt, black] (0,1.14) to (1.73,0); 

    \draw [bend right=10, line width=0.5pt, black] (0,1.14) to (-1.73,0); 
    \draw [bend left=10, line width=0.5pt, black] (0,1.14) to (-1.73,0);

    \draw [line width=0.5pt, black] (-1.73,0) to (1.73,0); 

    \draw [fill=white,line width=0.5pt] (0,3) node[above] {} circle (4pt) ; 
    \draw [fill=white,line width=0.5pt] (-1.73,0) node[left] {} circle (4pt) ; 
    \draw [fill=white,line width=0.5pt] (1.73,0) node[right] {} circle (4pt) ; 
    \draw [fill=white,line width=0.5pt] (0,1.14) node[below=2mm] {} circle (4pt); 
    
\end{tikzpicture}
        \caption{$2K_4-e$ }
		\label{fig:4-w=-2-graph}
    \end{subfigure}
    \caption{All the elements of $\mathcal{F}_1$}
    \label{fig:4-graphs-w(G)=-2}
\end{figure}
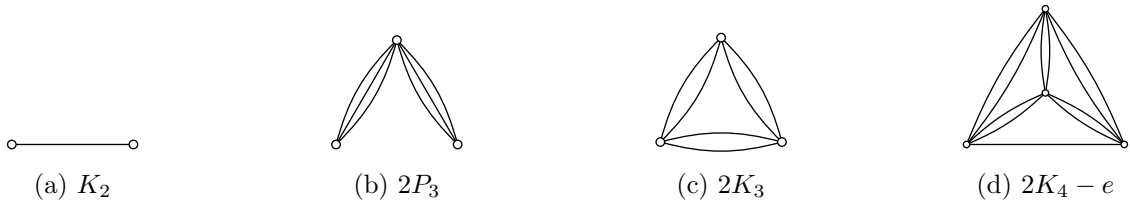

\begin{lemma}\label{lem:w(H)=-2k}
Fix $k\in \{1,2\}$. For each $H\subseteq G$, if $w(H)=-2k$, then either:
    \begin{enumerate}[label=(\arabic*)]
    \setlength{\itemsep}{0em}
        \item\label{-2kweight1} Every graph $H'$ formed from $H$ by adding $k$ edges, without creating an edge of multiplicity $5$, is $\S_5$-contractible; or
        \item The subgraph $H$ is in $\mathcal{F}_k$.
    \end{enumerate}
Moreover, if (1) holds, then $v(H)\ge 4$ and $H$ is $(6-k)$-edge-connected.
\end{lemma}

\begin{proof}
Fix $H \subset G$ with $w(H) = -2k$, so $e(H) = 5v(H) - (8+k)$. 

	If $v(H) = 2$, then $e(H) = 2-k$, and $H$ is formed from $2K_2$ by deleting $k$ edges; so (2) holds and (1) does not.
	If $v(H) = 3$, then $e(H) = 7-k$. Since $G$ contains neither $T_{1,1,3}$ (\Cref{lemma:no-T113}) nor $4K_2$ (it is $\S_5$-contractible), $H$ must be formed from $T_{1,3,3}$ or $T_{2,2,3}$ by deleting $k$ edges without creating a copy of $T_{1,1,3}$. So again (2) holds and (1) does not.

    Now assume $v(H) \ge 4$. Consider a partition $\P = \{X, X^c\}$ of $H$.  Since $v(H) \ge 4$, either $|X|, |X^c| \ge 2$, or $|X| \ge 3$ and $|X^c| = 1$. In both cases, by the Gap Lemma,
    \begin{align*}
	2|[X, X^c]| - 20 + 16 & = w_H(\P)  = \sigma(H[X]) + \sigma(H[X^c]) + w(H) \\
	& \ge \min\{4+4, 8+0\} - 2k = 8 - 2k.
    \end{align*} 
    So $|[X, X^c]| \ge 6-k$, i.e., $H$ is $(6-k)$-edge-connected.

    Assume that \ref{-2kweight1} fails to hold. Then there exist $k$ edges $e_1 = u_1u_2$ and $e_2 = v_1v_2$ (when $k=2$) such that $H' = H + e_1$ (if $k=1$) or $H' = H + e_1 + e_2$ (if $k=2$), and $H'$ contains no $5K_2$ and $H'$ is not $\S_5$-contractible.  

    For the trivial partition $\P_0$ of $H'$, by \Cref{lemma:gap-lemma} we have $w_{H'}(\P_0) = w(H) + 2k = -2k + 2k \ge 0$. Consider now a nontrivial partition $\P$ of $H'$. If $k=1$, then $w_{H'}(\P) \ge w_H(\P) \ge w(H) + 4 = 2$. When $k=2$, we consider $2$ cases. In the first case, either $e_1$ or $e_2$ lies in $H'/\P$, and since $\P$ is of type $(2^+, *)$, we have $w_{H'}(\P) \ge w_H(\P) + 2 \ge 4 - 4 + 2 = 2$. In the second case, neither $e_i$ lies in $H'/\P$. Then $\P$ is of type either $(2^+, 2^+, *)$ or $(3^+, *)$, or is of type $(2,1,1,\dots,1)$ in which case the 
    $2$ additional edges $e_1$ and $e_2$ are parallel between $u_1$ and $u_2$ (recalling that $e_1=u_1u_2$) but $\mu_G(u_1,u_2) \le 2$ as $H'$ contains no $5K_2$. In all these situations, we have $w_{H'}(\P) = w_H(\P) \ge \min\{4+4-4, 8-4, 6-4\} = 2$.

     Hence $w(H') \ge 0$, and every nontrivial partition $\P$ of $H'$ satisfies $w_{H'}(\P) \ge 2$. Since $H'$ is not $\S_5$-contractible, for some partition $\P'$ we have $H'/\P' \in \mathcal{N}_5$. Because $H$ is $(6-k)$-edge-connected, so is $H'$, implying $H'/\P' \in \{T_{1,3,3}, T_{2,2,3}, W_1, W_2\}$.  

     Every nontrivial partition $\P$ of $H'$ satisfies $w_{H'}(\P) \ge 2$, so by~\Cref{ob:weight-K2-Tabc}\ref{w-N5}, $\P'$ must be trivial. Thus $H'/\P' = H' \in \{T_{1,3,3}, T_{2,2,3}, W_1, W_2\}$. Since $v(H') \ge 4$, we have $H' \in \{W_1, W_2\}$.  

    Finally, $W_1$ contains $3$ copies of $T_{1,1,3}$ not sharing a common $3K_2$. So if $H' = W_1$, then $H$ contains $T_{1,1,3}$, a contradiction. Hence $H' = W_2$, and $H$ is formed from $H'$ by deleting $k$ edges. Since $H$ contains no $T_{1,1,3}$, it follows that $H$ is formed from $2K_4$ by deleting $k$ edges.
\end{proof}

For a subgraph $H$ of $G$, a path $P = v_0 v_1 \dots v_n$ is called an \emph{$(H,n)$-path} if $n \ge 2$, $v_0, v_n \in V(H)$, and all internal vertices $v_1, \dots, v_{n-1} \in V(G) \setminus V(H)$. For brevity, we may simply refer to $P$ as an $H$-path. Moreover, if $G$ has 2 $H$-paths $P_1 = u_0 u_1 \dots u_m$ and $P_2 = v_0 v_1 \dots v_n$ such that $\{u_1, \dots, u_{m-1}\} \cap \{v_1, \dots, v_{n-1}\} = \emptyset$, then $P_1$ and $P_2$ are called \emph{internally disjoint}.

Based on \Cref{lem:w(H)=-2k}, we can analyze the behavior of $H$-paths in $G$ for subgraphs $H$ with $w(H)=-2k$ that are not in $\mathcal{F}_k$.

\begin{lemma}\label{lem:w=-2-path}
Let $H$ be a subgraph of $G$ with $v(H) < v(G)$ and $w(H) = -2$. If $H \not\in \mathcal{F}_1$, then $H$ is an induced subgraph of $G$. Moreover, for each $(H,n)$-path in $G$, we have $n \ge 4$.
\end{lemma}

\begin{proof}
Since $w(H) = -2$ and $H \not\in \mathcal{F}_1$, \Cref{lem:w(H)=-2k} implies that adding any edge $e$ to $H$ produces a graph that is $\S_5$-contractible. This ensures that $H$ is induced; otherwise, we would contradict \Cref{lem:property}\ref{lem:property-subgraph}. Furthermore, by the moreover part of \Cref{lem:w(H)=-2k}, $H$ is $5$-edge-connected and has $v(H) \ge 4$.

Let $P$ be an $(H,n)$-path in $G$, denoted by $v_0 v_1 \dots v_n$. Since $H$ is induced, $n \ge 2$. Suppose, for contradiction, that $n \le 3$. Lift the path $P$ to create a new edge $e = v_0 v_n$ and let $H' = H + e$. Then $H'$ is $\S_5$-contractible. Contracting $H'$ to a single vertex $v_H$ yields a graph $G'$. For each partition $\P$ of $G'$, let $\P_H$ denote the corresponding $H$-restored partition of $G$. Since $v(H) \ge 4$, $\P_H$ is a $(4^+,*)$-partition of $G$, giving $w_G(\P_H) \ge 8$ by \Cref{lemma:gap-lemma}. Hence,
\(
w_{G'}(\P) \ge w_G(\P_H) - 2n \ge 8 - 2n \ge 8 - 6 = 2,
\)
which implies $w(G') \ge 2$. Since $G$ is $6$-edge-connected (\Cref{lem:6-edge-connected}), $G'$ is $4$-edge-connected. By \Cref{lem:property}\ref{lem:property-lifting}, $G'$ cannot be $\S_5$-contractible. Therefore, there exists a partition $\P'$ such that $G'/\P' \in \N_5$. However, as $w(G') \ge 2$, \Cref{ob:weight-K2-Tabc}\ref{w-N5} implies that $G'/\P' = 3K_2$. But this is impossible since $G'$ is $4$-edge-connected. This contradiction shows that $n \ge 4$.
\end{proof}

\begin{lemma}\label{lem:w=-4-path}
Let $H$ be a subgraph of $G$ with $v(H)<v(G)$, $w(H)=-4$, and $H \not\in \mathcal{F}_2$. 
If $G$ has an $(H,2)$-path $P$, then $H$ is an induced subgraph of $G$. Moreover, if $G$ has another $(H,n)$-path that is internally disjoint from $P$, then $n \geq 4$.
\end{lemma}

\begin{proof}
Since $w(H)=-4$ and $H \not\in \mathcal{F}_2$, by~\Cref{lem:w(H)=-2k} adding any $2$ edges to $H$ without creating $5K_2$ forms a graph that is $\S_5$-contractible. By the moreover part of~\Cref{lem:w(H)=-2k}, $H$ is $4$-edge-connected and $v(H)\ge 4$. Let $P$ be an $(H,2)$-path in $G$, denoted $u_0u_1u_2$. 

Assume instead that $H$ is not an induced subgraph of $G$. So $w(G[V(H)]) \ge -2$. If $w(G[V(H)])=-2$, then $H$ arises from $G[V(H)]$ by deleting a single edge. Since $H \not\in \mathcal{F}_2$, it follows that $G[V(H)] \not\in \mathcal{F}_1$. However, $P$ is also a $(G[V(H)],2)$-path, contradicting \Cref{lem:w=-2-path}. If $w(G[V(H)]) \ge 0$, then by \Cref{lemma:gap-lemma}, $v(H) \le 2$, a contradiction. Hence, $H$ must be an induced subgraph of $G$.

Let $Q$ be an $(H,n)$-path internally disjoint from $P$; denote $Q$ by $v_0v_1\dots v_n$. As $H$ is an induced subgraph, $n\ge 2$. Assume, contrary to the lemma, that $n \le 3$. We lift paths $P$ and $Q$ by adding edges $u_0u_2$ and $v_0v_n$, and let $H' := H + u_0u_2+ v_0v_n$. Since $G$ has neither $4K_2$ nor $T_{1,1,3}$, we cannot create $5K_2$. So by \Cref{lem:w(H)=-2k}\ref{-2kweight1}, $H'$ is $\S_5$-contractible.

We contract $H'$ to a single vertex $v_H$ and call the resulting graph $G'$. Note that $G'$ is not $\S_5$-contractible by~\Cref{lem:property}\ref{lem:property-lifting}. We will reach a contradiction by showing that in fact $G'$ is $\S_5$-contractible.
Specifically, we show that $w(G')\ge 0$ and that $w_{G'}(\mathcal{P})\ge 2$ for most partitions $\mathcal{P}$.
Moreover, $G'$ is $4$-edge-connected since $G$ is $6$-edge-connected by~\Cref{lem:6-edge-connected} and essentially $8$-edge-connected by~\Cref{lem:essential-8-edge-connected}. Let $\mathcal{P} = \{V_1,\dots,V_t\}$ be a partition of $G'$ with $v_H \in V_1$, and let $\mathcal{P}_H = \{V_1',\dots,V_t'\}$ be the corresponding $H$-restored partition of $G$ such that $V(H)\subseteq V_1'$. Since $v(H)\ge 4$, $\mathcal{P}_H$ has type $(4^+,*)$. 

If $\mathcal{P}$ is trivial, then since $w(H)=-4$ by hypothesis, we have
\[
w_{G'}(\mathcal{P}) = w_G(\mathcal{P}_H) - 2 \times (2+n) = (6-(-4)) - 4 - 2n = 6 - 2n.
\]
	If $\mathcal{P}$ is nontrivial and some part $V_i$ with $i\ne 1$ satisfies $|V_i| \ge 2$, then $\mathcal{P}_H$ has type $(4^+,2^+,*)$. So by \Cref{lemma:gap-lemma}(2),
\[
w_{G'}(\mathcal{P}) \ge w_G(\mathcal{P}_H) - 2 \times (2+n) \ge 8 + 4 - 4 - 2n = 8 - 2n \ge 2.
\]
Otherwise, $\mathcal{P}$ has type $(2^+,1,\dots,1)$ with $|V_1|\ge 2$, and hence $|V_1'| \ge v(H)+1 \ge 5$.

\textbf{Case 1: $\bm{n=2}$.} For each partition $\mathcal{P}$ of type $(2^+,1,1,\dots,1)$ with $|V_1|\ge 2$, 
we consider the placement of $u_1$ and $v_1$. If at least one of $u_1$ and $v_1$ lies in $V_1'$, then 
\[
w_{G'}(\mathcal{P}) \ge w_G(\mathcal{P}_H) - 2 \times 2 \ge 8-4 = 4.
\] 
If neither $u_1$ nor $v_1$ lies in $V_1'$, then since $|V_1'|\ge 5$ and $P$ is a $(G[V_1'],2)$-path, by~\Cref{lem:w=-2-path} we have $w(G[V_1']) \le -4$. It follows that 
\[
w_{G'}(\mathcal{P}) \ge w_G(\mathcal{P}_H) - 2 \times 4 = (6 - (-4)) - 8 = 2.
\]
In all cases, $w_{G'}(\mathcal{P}) \ge \min\{6-2n, 8-2n, 4, 2\} = 2$; hence $w(G') \ge 2$. Since $G'$ is not $\S_5$-contractible, there must exist a partition $\mathcal{P}'$ such that $G'/\mathcal{P}' \in \mathcal{N}_5$. This would require $G'/\mathcal{P}' = 3K_2$, which is impossible since $G'$ is $4$-edge-connected. Thus, $n=2$ cannot occur.

\textbf{Case 2: $\bm{n=3}$.} For any partition $\mathcal{P}$ of type $(2^+,1,1,\dots,1)$ with $|V_1|\ge 2$, we analyze the weight depending on the placement of $u_1,v_1,v_2$:
\begin{enumerate}[label=(\arabic*)]
    \item If $|\{u_1,v_1,v_2\}\cap V_1'|\ge 2$, then $w_{G'}(\mathcal{P}) \ge w_G(\mathcal{P}_H) -2\times 2\ge 8- 4 =4$;
    \item If $|\{u_1,v_1,v_2\}\cap V_1'| = 1$, then $G$ contains a $(G[V_1'],2)$-path or $(G[V_1'],3)$-path. 
  By~\Cref{lem:w=-2-path}, $w(G[V_1']) \le -4$, yielding $w_{G'}(\mathcal{P}) \ge (6-(-4)) - 2\times 4 = 2$;
    \item If none of $u_1,v_1,v_2$ lies in $V_1'$, then $G$ contains a $(G[V_1'],2)$-path $u_0u_1u_2$, and again by~\Cref{lem:w=-2-path}, $w(G[V_1']) \le -4$, giving $w_{G'}(\mathcal{P}) \ge (6-(-4)) - 2\times 5 = 0$.
\end{enumerate}
	In all cases, $w_{G'}(\mathcal{P}) \ge \min\{6-2n, 8-2n, 4, 2, 0\} = 0$; hence $w(G') \ge 0$. Since $G'$ is not $\S_5$-contractible, there exists a partition $\mathcal{P}'$ with $G'/\mathcal{P}' \in \mathcal{N}_5$. As $G'$ is $4$-edge-connected, $G'/\mathcal{P}' \in \mathcal{N}_5 \setminus \{3K_2,2K_2\}$, implying $w_{G'}(\mathcal{P}')=0$. Since $v(H)\ge 4$ and $w_{G'}(\mathcal{P}')=0$, the above analysis implies that the $H$-restored partition $\mathcal{P}_H' = \{V_1',V_2',\dots,V_t'\}$ is of type $(4^+,1,1,\dots,1)$, with $V(H)\subseteq V_1'$, and none of $u_1,v_1,v_2$ lies in $V_1'$. Thus, $v(G'/\mathcal{P}') \ge 1+|\{u_1,v_1,v_2\}| = 4$, and $G'/\mathcal{P}'\in\{W_1,W_2\}$. Since $G$ does not contain $T_{1,1,3}$ by \Cref{lemma:no-T113}, the $9$-vertex in $W_1$, or one of the $7$-vertices in $W_2$, say $x$, must correspond to $V_1'$, with the $3$ other vertices corresponding to $u_1,v_1,v_2$.

If $G'/\mathcal{P}' = W_1$, or $G'/\mathcal{P}' = W_2$ and $u_1$ corresponds to a $5$-vertex, then at least $2$ of $\{u_1,v_1,v_2\}$, say $x_1$ and $x_2$, each sends at least $4$ edges to $V_1'$ in $G$. Since $G$ contains no $4K_2$, this implies $2$ internally disjoint $(G[V_1'],2)$-paths. By Case 1 above, no $2$ internally disjoint $(G[V_1'],2)$-paths exist when $w(G[V_1'])=-4$ and $G[V_1'] \not\in \mathcal{F}_2$. Hence, $w(G[V_1']) \le -6$. This gives $w_{G'}(\mathcal{P}') \ge (6-(-6)) - 10 = 2$, contradicting $w_{G'}(\mathcal{P}')=0$.

Hence, $G'/\mathcal{P}' = W_2$ and $u_1$ corresponds to the other $7$-vertex. Now the structure of $G'/\mathcal{P}'$ and $G$ must be as shown in \Cref{fig:14-E2-G'-G}. 
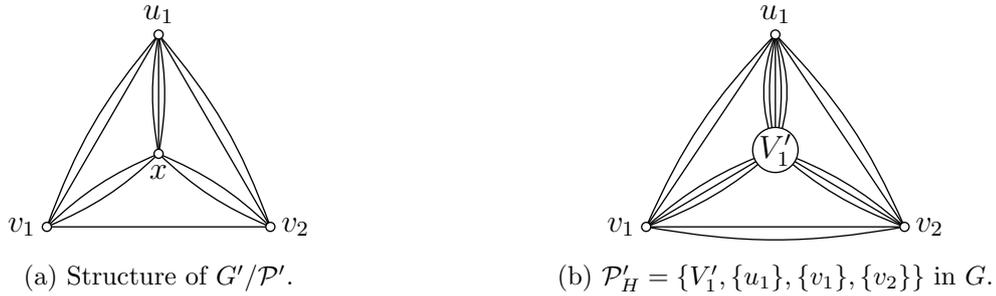
\begin{figure}[ht]
    \centering
    \begin{subfigure}{0.48\textwidth}
    \centering
    \begin{tikzpicture}[scale=0.85]
    \draw [line width=0.5pt, black] (0,3) to (0,1.14); 
    \draw [bend left=10,line width=0.5pt, black] (0,3) to (0,1.14); 
    \draw [bend right=10, line width=0.5pt, black] (0,3) to (0,1.14); 
    \draw [bend right=10, line width=0.5pt, black] (0,3) to (-1.73,0); 
    \draw [bend left=0, line width=0.5pt, black] (0,3) to (-1.73,0); 
    \draw [bend right=0, line width=0.5pt, black] (0,3) to (1.73,0); 
    \draw [bend left=10, line width=0.5pt, black] (0,3) to (1.73,0); 
    \draw [bend right=10, line width=0.5pt, black] (0,1.14) to (1.73,0); 
    \draw [bend left=10, line width=0.5pt, black] (0,1.14) to (1.73,0);  
    \draw [bend right=10, line width=0.5pt, black] (0,1.14) to (-1.73,0); 
    \draw [bend left=10, line width=0.5pt, black] (0,1.14) to (-1.73,0); 
    \draw [line width=0.5pt, black] (-1.73,0) to (1.73,0); 
    \draw [fill=white,line width=0.5pt] (0,3) node[above] {$u_1$} circle (4pt); 
    \draw [fill=white,line width=0.5pt] (-1.73,0) node[left] {$v_1$} circle (4pt); 
    \draw [fill=white,line width=0.5pt] (1.73,0) node[right] {$v_2$} circle (4pt); 
    \draw [fill=white,line width=0.5pt] (0,1.14) node[below] {$x$} circle (4pt);        
    \end{tikzpicture}
    \caption{Structure of $G'/\mathcal{P}'$.}
    \label{fig:14-E2-G}
    \end{subfigure}
    \hspace{0.01\textwidth}
    \begin{subfigure}{0.48\textwidth}
        \centering
        \begin{tikzpicture}[scale=0.85]
        \coordinate (u1) at (0,3);
        \coordinate (v1) at (-2,0);
        \coordinate (v2) at (2,0);
        \coordinate (V1') at (0,1.2);
        \foreach \angle in {-20,-10,0,10,20} {
            \draw [bend left=\angle, line width=0.5pt] (u1) to (V1');
        }
        \draw [bend right=10, line width=0.5pt] (u1) to (v1);
        \draw [bend left=0, line width=0.5pt] (u1) to (v1);
        \draw [bend right=0, line width=0.5pt] (u1) to (v2);
        \draw [bend left=10, line width=0.5pt] (u1) to (v2);
        \foreach \target in {v1, v2} {
            \draw [line width=0.5pt] (V1') -- (\target);
            \draw [bend left=10, line width=0.5pt] (V1') to (\target);
            \draw [bend right=10, line width=0.5pt] (V1') to (\target);
        }
        \draw [line width=0.5pt] (v1) -- (v2);
        \draw [bend right=10, line width=0.5pt] (v1) to (v2);
        \draw [fill=white,line width=0.5pt] (u1) node[above] {$u_1$} circle (4pt);
        \draw [fill=white,line width=0.5pt] (v1) node[left] {$v_1$} circle (4pt);
        \draw [fill=white,line width=0.5pt] (v2) node[right] {$v_2$} circle (4pt);
        \draw [fill=white,line width=0.5pt] (V1') node[below=-3.5mm] {$V_1'$} circle (10pt);
        \end{tikzpicture}
        \caption{$\mathcal{P}_H' = \{V_1', \{u_1\}, \{v_1\}, \{v_2\}\}$ in $G$.}
        \label{fig:14-E2-G'}
    \end{subfigure}
    \caption{The graphs $G'/\mathcal{P}'$ and $G$.}
    \label{fig:14-E2-G'-G}
\end{figure}

Note that $|V_1'|<v(G)$ and $\mathcal{P}_H' = \{V_1',\{u_1\},\{v_1\},\{v_2\}\}$. 
Also,
\[
w_G(\mathcal{P}_H') = 6 - w(G[V_1']) = 2 \times 17 - 10 \times 4 + 16 = 10.
\]
Thus, $w(G[V_1']) = -4$. So either $G[V_1']=H$ or $|V_1'|>v(H)$, and both show that $G[V_1'] \not\in \mathcal{F}_2$. By~\Cref{lem:w(H)=-2k}\ref{-2kweight1}, adding any $2$ edges to $G[V_1']$ without creating $5K_2$ yields an $\S_5$-contractible graph. Since $|[u_1,V_1']| = 5$ and $G$ contains no $4K_2$, $u_1$ has at least two neighbors in $V_1'$, yielding two edge-disjoint $2$-paths $y_1 u_1 y_2$ and $y_1' u_1 y_2'$ with $y_1,y_2,y_1',y_2' \in V_1'$. Lifting these paths to edges $y_1y_2$ and $y_1'y_2'$ produces $G_1$, which contains no $5K_2$; so $G_1[V_1']$ is $\S_5$-contractible by \Cref{lem:w(H)=-2k}. Contracting $G_1[V_1']$ gives $G_2$ with $v(G_2)=4$, $e(G_2)=13$, $\mu(G_2)\le3$, and $\delta(G_2)\ge 5$. By~\Cref{thm:small-S5-contractible}\ref{3_thm:small-S5-contractible}, $G_2$ is $\S_5$-contractible, contradicting \Cref{lem:property}\ref{lem:property-lifting}. Hence, $n\ne3$.
\end{proof}

\begin{lemma}\label{lemma:7-edge-conn}
The graph $G$ is $7$-edge-connected.
\end{lemma}

\begin{proof}
By \Cref{lem:6-edge-connected} and \Cref{lem:essential-8-edge-connected}, it suffices to show that $G$ has no 
$6$-vertex. Assume instead that $v\in V(G)$ is a $6$-vertex. Let $G_1 := G - v$. Since $w(G) = 0$ by~\Cref{cor:w(G)=0}, \Cref{lem:trivial} gives
\[
w(G_1) = w(G) - 2 \times 6 + 10 \times 1 = -2.
\]
By \Cref{lem:w=-2-path}, as there is a $(G_1,2)$-path via $v$, we have $G_1 \in \mathcal{F}_1$. Since $v(G_1) = v(G) - 1 \ge 4$, $G_1$ must be formed from $2K_4$ by deleting one edge, as shown in \Cref{fig:4-w=-2-graph}. Denote by $v_1$ and $v_2$ the two $6$-vertices in $G_1$, and by $v_3$ and $v_4$ the two $5$-vertices in $G_1$.

Since $\delta(G)\ge6$ by~\Cref{lem:6-edge-connected}, the vertex $v$ must be adjacent to both $v_3$ and $v_4$ in $G$. Moreover, $v$ cannot be adjacent to both of $v_1$ and $v_2$ in $G$, as otherwise $G$ would contain a $K_5$, contradicting the planarity of $G$. By symmetry, we may assume that $v$ is not adjacent to $v_1$. Since $G$ contains no $T_{1,1,3}$ by~\Cref{lemma:no-T113}, we must have
\(
\mu_G(v,v_2)=\mu_G(v,v_3)=\mu_G(v,v_4)=2.
\)
Therefore, $G$ is the graph as shown in \Cref{fig:7-edge-conn}.

\begin{figure}[ht!]
    \centering
    \begin{tikzpicture}[scale=0.8]			
        \draw [bend left=10,line width=0.5pt] (0,3) to (0,1.5); 
        \draw [bend right=10,line width=0.5pt] (0,3) to (0,1.5); 
        \draw [bend left=10,line width=0.5pt] (0,0.6) to (0,1.5); 
        \draw [bend right=10,line width=0.5pt] (0,0.6) to (0,1.5); 
        \draw [bend right=5,line width=0.5pt] (0,3) to (-1.73,0); 
        \draw [bend left=5,line width=0.5pt] (0,3) to (-1.73,0); 
        \draw [bend right=5,line width=0.5pt] (0,3) to (1.73,0); 
        \draw [bend left=5,line width=0.5pt] (0,3) to (1.73,0); 
        \draw [bend right=6,line width=0.5pt] (0,1.5) to (1.73,0); 
        \draw [bend left=6,line width=0.5pt] (0,1.5) to (1.73,0); 
        \draw [bend right=7,line width=0.5pt] (0,0.6) to (1.73,0); 
        \draw [bend left=7,line width=0.5pt] (0,0.6) to (1.73,0); 
        \draw [bend right=6,line width=0.5pt] (0,1.5) to (-1.73,0); 
        \draw [bend left=6,line width=0.5pt] (0,1.5) to (-1.73,0); 
        \draw [bend right=7,line width=0.5pt] (0,0.6) to (-1.73,0); 
        \draw [bend left=7,line width=0.5pt] (0,0.6) to (-1.73,0); 
        \draw [line width=0.5pt] (-1.73,0) to (1.73,0); 

        \draw [fill=white,line width=0.5pt] (0,3) node[above] {$v_1$} circle (3.5pt); 
        \draw [fill=white,line width=0.5pt] (-1.73,0) node[left] {$v_3$} circle (3.5pt); 
        \draw [fill=white,line width=0.5pt] (1.73,0) node[right] {$v_4$} circle (3.5pt); 
        \draw [fill=white,line width=0.5pt] (0,1.5) node[left=0.1mm] {$v_2$} circle (3.5pt); 
        \draw [fill=white,line width=0.5pt] (0,0.6) node[below=1mm] {$v$} circle (3.5pt); 
    \end{tikzpicture}
    \caption{The only possible graph $G$ in \Cref{lemma:7-edge-conn}.}
    \label{fig:7-edge-conn}
\end{figure}
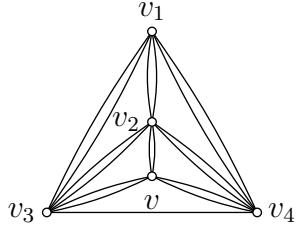

	We now lift the $(G_1,2)$-path $v_3 v v_4$. The resulting subgraph $H$ induced by $\{v_1,v_2,v_3,v_4\}$ is $2K_4$, which is $S_5$-contractible by \Cref{thm:small-S5-contractible}\ref{3_thm:small-S5-contractible}. Contracting $H$ yields $4K_2$, also $S_5$-contractible, contradicting \Cref{lem:property}\ref{lem:property-lifting}. So $\delta(G)\ge 7$ and $G$ is $7$-edge-connected.
\end{proof}

\begin{lemma}\label{lemma:v(G)}
    The graph $G$ satisfies $v(G) \ge 6$.
\end{lemma}

\begin{proof}
Assume instead that $v(G) = 5$. By \Cref{cor:w(G)=0}, we have
\(
e(G) = 5v(G) - 8 = 17.
\)
On the other hand, \Cref{lemma:7-edge-conn} implies $\delta(G) \ge 7$, and thus
\(
e(G) \ge \frac{7 v(G)}{2} = \frac{35}{2} = 17.5,
\)
which is impossible. This contradiction shows that $v(G) \ge 6$.
\end{proof}

\begin{lemma}\label{lemma:essential-10}
    The graph $G$ is essentially $10$-edge-connected.
\end{lemma}
\begin{proof}
Let $[X_1,X_2]$ be an edge-cut of $G$ with $|X_2|\ge |X_1|\ge 2$, and let $\mathcal{P}=\{X_1,X_2\}$. 
By \Cref{lem:essential-8-edge-connected}, we have $|[X_1,X_2]|\ge 8$.
Since $v(G)\ge 6$ by \Cref{lemma:v(G)} and $|X_2|\ge |X_1|$, it follows that $|X_2|\ge 3$, and hence $\sigma(G[X_2])\ge 8$.

If $|X_1|\ge 3$, then by \Cref{lemma:gap-lemma} we also have $\sigma(G[X_1])\ge 8$. Thus \[
w_G(\mathcal{P})
= \sigma(G[X_1])+\sigma(G[X_2])\ge 16.
\]
Moreover, $w_G(\mathcal{P})= 2|[X_1,X_2]|-10\times 2+16$, and then $|[X_1,X_2]|\ge 10$. Hence, we assume $|X_1|=2$.

If $G[X_1]\in\{K_2,2K_2\}$, then since $\delta(G)\ge 7$ by \Cref{lemma:7-edge-conn}, we obtain
$|[X_1,X_2]|\ge 2\delta(G)-2|E(G[X_1])| \ge 10$.
As $4K_2\not\subseteq G$, the only remaining case is $G[X_1]=3K_2$.

Assume instead that $|[X_1,X_2]|\le 9$. Denote $X_1$ by $\{u,v\}$, and let $G_1:=G-\{u,v\}$.  
Since $d_G(u)\ge 7$ and $d_G(v)\ge 7$, and $G$ is essentially $8$-edge-connected, we have either
$d_G(u)=d_G(v)=7$ or $d_G(u)=7$ and $d_G(v)=8$, where we assume $d_G(u)\le d_G(v)$.

Observe that
$w(G_1)
= w(G)-2\bigl(d_G(u)+d_G(v)-3\bigr)+10\times 2
\in\{-4,-2\}$.
Moreover, $|[u,V(G_1)]|\ge 4$ and $|[v,V(G_1)]|\ge 4$; this implies that there exist $2$ internally disjoint $(G_1,2)$-paths through $u$ and $v$. Since $w(G_1)\ge -4$, Lemmas~\ref{lem:w=-2-path} and~\ref{lem:w=-4-path} yield $G_1\in \mathcal{F}_1\cup\mathcal{F}_2$.

Furthermore, since $v(G_1)=v(G)-2\ge 4$ by \Cref{lemma:v(G)}, the graph $G_1$ must be formed from $2K_4$ by 
deleting $1$ or $2$ edges. In particular, $\Delta(G_1)\le 6$.

If $G_1$ contains a copy of $K_4$, then each $x\in V(G_1)$ satisfies $d_{G_1}(x)\le 6$ but $\delta(G)\ge 7$; 
hence each such $x$ must be adjacent to either $u$ or $v$. Contracting $G[\{u,v\}]$ in $G$ yields a $K_5$, contradicting that $G$ is planar. Therefore, $G_1=F^*$, where $F^*$ is the graph in $\mathcal{F}_2$ consisting of two copies of $T_{2,2,2}$ sharing a common $2K_2$; equivalently, $F^*$ is formed from $2K_4$ by deleting an edge with multiplicity $2$. In particular, $w(G_1)=-4$, which implies $d_G(u)=7$ and $d_G(v)=8$.

Since $G$ contains no $4K_2$, each of $u$ and $v$ has at least $2$ neighbors in $G_1$. Moreover, as $G$ contains no $T_{1,3,3}$ by \Cref{lemma:no-T113}, vertices $u$ and $v$ have no common neighbor in $G_1$. Hence, each of them has exactly $2$ neighbors in $G_1$.

Again by \Cref{lemma:no-T113}, since $|[\{v\},V(G_1)]|=5$, the two neighbors of $v$ in $G_1$ cannot be adjacent; otherwise, $G$ contains a $T_{1,1,3}$. So $u$ must be adjacent to two $6$-vertices of $G_1$, and 
$v$ must be adjacent to two $4$-vertices. Thus, $G$ contains $K_{3,3}$ as a subgraph, contradicting that $G$ is planar.
This contradiction completes the proof.
\end{proof}

\begin{lemma}\label{lemma:T222-forbidden}
Let $H$ be a subgraph of $G$. If $H$ is $T_{2,2,2}$, then there do not exist $2$ internally disjoint $(H,2)$-paths in $G$.
\end{lemma}

\begin{proof}
Let $H$ be a copy of $T_{2,2,2}$ in $G$, and denote $V(H)$ by $\{x,y,z\}$. 
Note that $H$ is an induced subgraph of $G$. Suppose, instead, that $G$ contains two internally disjoint $(H,2)$-paths $P$ and $Q$, containing vertices $u$ and $v$, respectively, as shown in \Cref{fig:T222-forbidden}. By symmetry, there are two possible cases: either $u$ is adjacent to $x,y$ and $v$ is adjacent to $x,z$, or both $u$ and $v$ are adjacent to each of $x$ and $y$.
We lift paths $P$ and $Q$. After lifting, the subgraph $H'$ induced by $V(H)$ is either $T_{2,3,3}$ or $T_{2,2,4}$, both of which are $\S_5$-contractible by \Cref{thm:small-S5-contractible}\ref{2_thm:small-S5-contractible}. Contracting $H'$ to a single vertex $v_H$ yields a graph $G'$.

\begin{figure}[htbp!]
    \centering
    \begin{subfigure}{0.48\textwidth}
    \centering
        \begin{tikzpicture}[scale=0.8]			
    \draw [bend left=15, line width=0.5pt, black] (-1,0) to (0,1.73); 
    \draw [bend left=15, line width=0.5pt, black] (0,1.73) to (-1,0); 

    \draw [bend left=15, line width=0.5pt, black] (1,0) to (0,1.73); 
    \draw [bend left=15, line width=0.5pt, black] (0,1.73) to (1,0); 

    \draw [bend left=15, line width=0.5pt, black] (-1,0) to (1,0); 
    \draw [bend right=15, line width=0.5pt, black] (-1,0) to (1,0); 

    \draw [line width=0.5pt, black] (-2,1.73) to (-1,0); 
    \draw [line width=0.5pt, black] (-2,1.73) to (0,1.73); 

    \draw [line width=0.5pt, black] (2,1.73) to (1,0); 
    \draw [line width=0.5pt, black] (2,1.73) to (0,1.73); 
    
    \draw [fill=white,line width=0.5pt] (-1,0) node[left] {$y$} circle (3.5pt) ; 
    \draw [fill=white,line width=0.5pt] (1,0) node[right] {$z$} circle (3.5pt) ;
    \draw [fill=white,line width=0.5pt] (0,1.73) node[above] {$x$} circle (3.5pt) ;
    \draw [fill=white,line width=0.5pt] (-2,1.73) node[left] {$u$} circle (3.5pt) ;
    \draw [fill=white,line width=0.5pt] (2,1.73) node[right] {$v$} circle (3.5pt) ;
\end{tikzpicture}
    \end{subfigure}
    \hfill
    \begin{subfigure}{0.48\textwidth}
    \centering
        \begin{tikzpicture}[scale=0.8]			
    \draw [bend left=15, line width=0.5pt, black] (-1,0) to (0,1.73); 
    \draw [bend left=15, line width=0.5pt, black] (0,1.73) to (-1,0); 

    \draw [bend left=15, line width=0.5pt, black] (1,0) to (0,1.73); 
    \draw [bend left=15, line width=0.5pt, black] (0,1.73) to (1,0); 

    \draw [bend left=15, line width=0.5pt, black] (-1,0) to (1,0); 
    \draw [bend right=15, line width=0.5pt, black] (-1,0) to (1,0); 

    \draw [line width=0.5pt, black] (-2,1.73) to (-1,0); 
    \draw [line width=0.5pt, black] (-2,1.73) to (0,1.73); 

    \draw [line width=0.5pt, black] (-1, 1.115)to (-1,0); 
    \draw [line width=0.5pt, black] (-1, 1.115)to (0,1.73); 
    
    \draw [fill=white,line width=0.5pt] (-1,0) node[left] {$y$} circle (3.5pt) ; 
    \draw [fill=white,line width=0.5pt] (1,0) node[right] {$z$} circle (3.5pt) ;
    \draw [fill=white,line width=0.5pt] (0,1.73) node[above] {$x$} circle (3.5pt) ;
    \draw [fill=white,line width=0.5pt] (-2,1.73) node[left] {$u$} circle (3.5pt) ;
    \draw [fill=white,line width=0.5pt] (-1, 1.115) node[left] {$v$} circle (3.5pt) ;
\end{tikzpicture}
    \end{subfigure}
    \caption{Forbidden subgraphs when $H\cong T_{2,2,2}$ where $V(H)=\{x,y,z\}$.}
    \label{fig:T222-forbidden}
\end{figure}
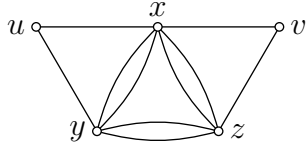

Let $\mathcal{P}=\{V_1,\dots,V_t\}$ be a partition of $G'$, and let $\mathcal{P}_H=\{V_1',\dots,V_t'\}$ be the corresponding $H$-restored partition of $G$, where $v_H\in V_1$ and $V(H)\subseteq V_1'$. Since $v(H)=3$, the partition $\mathcal{P}_H$ has type $(3^+,*)$. We analyze the possible values of $w_{G'}(\mathcal{P})$.

First, if $\mathcal{P}$ is trivial, then $w_G(\mathcal{P}_H)=\sigma(H)=6-(-2)=8$. Since lifting removes $4$ edges from $G/\mathcal{P}_H$, we have $w_{G'}(\mathcal{P})=8-2\times4=0$.

Next, suppose that $\mathcal{P}$ is nontrivial and that $\mathcal{P}_H$ has type $(3^+,2^+,*)$. Then $w_G(\mathcal{P}_H)\ge 8+4=12$, and again at most $4$ edges are removed. Hence $w_{G'}(\mathcal{P})\ge 12-2\times4=4$.

Finally, suppose that $\mathcal{P}_H$ has type $(4^+,1,\dots,1)$. If $|\{u,v\}\cap V_1'|\ge 1$, then at most $2$ edges are removed, so $w_{G'}(\mathcal{P})\ge 8-2\times2=4$. If $|\{u,v\}\cap V_1'|=0$, then $4$ edges are removed and $w_{G'}(\mathcal{P})\ge 8-2\times4=0$. So always $w_{G'}(\mathcal{P})\ge 0$. 
Thus, $w(G')\ge 0$.

Since $G'$ is not $\S_5$-contractible by \Cref{lem:property}\ref{lem:property-lifting}, there exists a partition $\mathcal{P}'$ such that $G'/\mathcal{P}'\in\mathcal{N}_5$. As $G$ is $7$-edge-connected by \Cref{lemma:7-edge-conn} and essentially $10$-edge-connected by \Cref{lemma:essential-10}, the graph $G'$ is $5$-edge-connected. Hence $G'/\mathcal{P}'\in\{W_1,W_2\}$, and therefore $w_{G'}(\mathcal{P}')=0$. By the analysis above, the corresponding partition $\mathcal{P}_H'$ must have type $(3^+,1,1,1)$, and also $u,v\notin V_1'$.

Suppose first that $G'/\mathcal{P}'=W_1$. Since $G$ has no $T_{1,1,3}$ by \Cref{lemma:no-T113}, the unique $9$-vertex of $W_1$ must correspond to $V_1'$, and two of the three $5$-vertices must correspond to $u$ and $v$. The final $5$-vertex must correspond to a $5$-vertex of $G$, contradicting $\delta(G)\ge 7$.

Suppose instead that $G'/\mathcal{P}'=W_2$, where $V(W_2)=\{v_1,v_2,v_3,v_4\}$ with $d(v_1)=d(v_2)=7$ and $d(v_3)=d(v_4)=5$. Since $G$ contains no $T_{1,1,3}$, one of $v_1$ and $v_2$, say $v_1$, corresponds to $V_1'$, while $v_3$ and $v_4$ correspond to $u$ and $v$, respectively. 
Thus $d_G(v_2)=d_G(v_3)=d_G(v_4)=7$. Let $H':=G[V_1']$. Then $w(H')=w(G)-2\times(12+4)+10\times3=-2$.
Since $P$ is also an $(H',2)$-path and $v(H')\ge v(H)=3$, it follows from \Cref{lem:w=-2-path} that $H'$ is 
either $T_{2,2,2}$ (i.e., $H'=H$) or the graph $2K_4-e$ shown in \Cref{fig:4-w=-2-graph}.

In the latter case, every vertex $v'\in V(H')$ satisfies $d_{H'}(v')\le 6$, and hence must be adjacent to at least one of $v_2,v_3,v_4$. Contracting $G[\{v_2,v_3,v_4\}]$ thus produces a $K_5$-minor, contradicting that $G$ is planar. Therefore, $H'=H$.  So $V(G)= \{x,y,z,v_2,v_3,v_4\}$.

Recall that $V(H)=\{x,y,z\}$. Note that $u$ and $v$, i.e., $v_3$ and $v_4$, each have $4$ edges connected to $H$. Since $4K_2$ is $\S_5$-contractible (which thus cannot appear in $G$), each of $v_3$ and $v_4$ must have at least $2$ neighbors in $H$.

First, consider the case where $v_3$ and $v_4$ are both adjacent to each of $x$ and $y$; this is illustrated on the right-hand side of \Cref{fig:T222-forbidden}. If $v_2$ is adjacent to $z$, then contracting the edge $v_2z$ induces a $K_5$-minor. Otherwise, if $v_2$ is adjacent to both $x$ and $y$, then contracting the edge $xz$ induces a $K_5$-minor. Finally, if $v_2$ is adjacent to exactly one of $x$ and $y$, then $G$ contains a $T_{1,1,3}$. Therefore, all cases are impossible, as $G$ is planar and contains no $T_{1,1,3}$.

Now we consider the case where $v_3$ 
is adjacent to $x$ and $y$, and $v_4$ is adjacent to $x$ and $z$; this is illustrated on the left-hand side of \Cref{fig:T222-forbidden}.
Since $G$ contains no $T_{1,1,3}$ and $|[\{v_2\},\{x,y,z\}]|=3$, the vertex $v_2$ must have at least $2$ neighbors in $\{x,y,z\}$.  If $x\in N(v_2)$, then contracting $G[\{y,z\}]$ gives a $K_5$-minor, contradicting that $G$ is planar. Thus, $x\notin N(v_2)$, and hence $y,z\in N(v_2)$.

Moreover, if one of $\{v_2,v_3,v_4\}$, say $v_i$, is adjacent to all vertices of $\{x,y,z\}$, then contracting $G[\{v_2,v_3,v_4\}\setminus\{v_i\}]$, gives a $K_5$-minor, again contradicting that $G$ is planar.
So each of $\{v_2,v_3,v_4\}$ has exactly $2$ neighbors in $\{x,y,z\}$.
Since $\delta(G)\ge 7$ and $G$ has no $T_{1,1,3}$, we have $\mu(v_3,y)=\mu(v_3,x)=\mu(v_4,z)=\mu(v_4,x)=2$ and $\{\mu(v_2,y),\mu(v_2,z)\}=\{1,2\}$.

Now we can lift the paths $yv_2z$ and $yv_3x$. The resulting graph contains an induced $T_{2,3,3}$ on $\{x,y,z\}$, whose contraction yields a graph containing $4K_2$ as a subgraph.
Contracting this $4K_2$ produces a copy of $T_{2,3,3}$.
Since both $T_{2,3,3}$ and $4K_2$ are $\S_5$-contractible, this contradicts \Cref{lem:property}\ref{lem:property-lifting}. This contradiction finishes the proof.
\end{proof}

Let $Q_{a_1,a_2,a_3,a_4}$ denote the graph with vertex set $\{v_1,v_2,v_3,v_4\}$, where the edge between $v_i$ and $v_{i+1}$ (with indices taken modulo $4$) has multiplicity $a_i$ for all $i\in\{1\ldots,4\}$, and all other vertex pairs $\{v_i,v_j\}$ have multiplicity zero. See \Cref{fig:4-graphs-Qabcd} for examples.

\begin{figure}[ht]
    \centering
    \begin{subfigure}{0.2\textwidth}
    \centering
        \begin{tikzpicture}[scale=1]			
    \draw [bend left=15, line width=0.5pt, black] (-1,1) to (1,1); 
    \draw [bend right=15, line width=0.5pt, black] (-1,1) to (1,1);

    \draw [bend left=15, line width=0.5pt, black] (1,-1) to (1,1); 
    \draw [bend right=15, line width=0.5pt, black] (1,-1) to (1,1);
    \draw [line width=0.5pt, black] (1,-1) to (1,1);

    \draw [bend left=15, line width=0.5pt, black] (1,-1) to (-1,-1); 
    \draw [bend right=15, line width=0.5pt, black] (1,-1) to (-1,-1);
    \draw [line width=0.5pt, black] (1,-1) to (-1,-1);

    \draw [bend left=15, line width=0.5pt, black] (-1,1) to (-1,-1); 
    \draw [bend right=15, line width=0.5pt, black] (-1,1) to (-1,-1);
    \draw [line width=0.5pt, black] (-1,1) to (-1,-1);
    
    \draw [fill=white,line width=0.5pt] (-1,1) node[left] {} circle (3.5pt) ; 
    \draw [fill=white,line width=0.5pt] (1,1) node[right] {} circle (3.5pt) ; 
    \draw [fill=white,line width=0.5pt] (-1,-1) node[left] {} circle (3.5pt) ; 
    \draw [fill=white,line width=0.5pt] (1,-1) node[right] {} circle (3.5pt) ;
\end{tikzpicture}
        \caption{$Q_{2,3,3,3}$}
		\label{fig:Q2333}
    \end{subfigure}
    \hfill
    \begin{subfigure}{0.2\textwidth}
    \centering
        \begin{tikzpicture}[scale=1]			
    \draw [bend left=15, line width=0.5pt, black] (-1,1) to (1,1); 
    \draw [bend right=15, line width=0.5pt, black] (-1,1) to (1,1);

    \draw [bend left=15, line width=0.5pt, black] (1,-1) to (1,1); 
    \draw [bend right=15, line width=0.5pt, black] (1,-1) to (1,1);

    \draw [bend left=15, line width=0.5pt, black] (1,-1) to (-1,-1); 
    \draw [bend right=15, line width=0.5pt, black] (1,-1) to (-1,-1);
    \draw [line width=0.5pt, black] (1,-1) to (-1,-1);

    \draw [bend left=15, line width=0.5pt, black] (-1,1) to (-1,-1); 
    \draw [bend right=15, line width=0.5pt, black] (-1,1) to (-1,-1);
    \draw [line width=0.5pt, black] (-1,1) to (-1,-1);
    
    \draw [fill=white,line width=0.5pt] (-1,1) node[left] {} circle (3.5pt) ; 
    \draw [fill=white,line width=0.5pt] (1,1) node[right] {} circle (3.5pt) ; 
    \draw [fill=white,line width=0.5pt] (-1,-1) node[left] {} circle (3.5pt) ; 
    \draw [fill=white,line width=0.5pt] (1,-1) node[right] {} circle (3.5pt) ;
\end{tikzpicture}
        \caption{$Q_{2,2,3,3}$}
		\label{fig:Q2233}
    \end{subfigure}
    \hfill
    \begin{subfigure}{0.2\textwidth}
    \centering
        \begin{tikzpicture}[scale=1]			
    \draw [bend left=15, line width=0.5pt, black] (-1,1) to (1,1); 
    \draw [bend right=15, line width=0.5pt, black] (-1,1) to (1,1);

    \draw [bend left=15, line width=0.5pt, black] (1,-1) to (1,1); 
    \draw [bend right=15, line width=0.5pt, black] (1,-1) to (1,1);
    \draw [line width=0.5pt, black] (1,-1) to (1,1);

    \draw [bend left=15, line width=0.5pt, black] (1,-1) to (-1,-1); 
    \draw [bend right=15, line width=0.5pt, black] (1,-1) to (-1,-1);

    \draw [bend left=15, line width=0.5pt, black] (-1,1) to (-1,-1); 
    \draw [bend right=15, line width=0.5pt, black] (-1,1) to (-1,-1);
    \draw [line width=0.5pt, black] (-1,1) to (-1,-1);
    
    \draw [fill=white,line width=0.5pt] (-1,1) node[left] {} circle (3.5pt) ; 
    \draw [fill=white,line width=0.5pt] (1,1) node[right] {} circle (3.5pt) ; 
    \draw [fill=white,line width=0.5pt] (-1,-1) node[left] {} circle (3.5pt) ; 
    \draw [fill=white,line width=0.5pt] (1,-1) node[right] {} circle (3.5pt) ;
\end{tikzpicture}
        \caption{$Q_{2,3,2,3}$}
		\label{fig:Q2323}
    \end{subfigure}
    \hfill
    \begin{subfigure}{0.2\textwidth}
    \centering
       \begin{tikzpicture}[scale=1]			
    \draw [bend left=15, line width=0.5pt, black] (-1,1) to (1,1); 
    \draw [bend right=15, line width=0.5pt, black] (-1,1) to (1,1);

    \draw [bend left=15, line width=0.5pt, black] (1,-1) to (1,1); 
    \draw [bend right=15, line width=0.5pt, black] (1,-1) to (1,1);

    \draw [bend left=15, line width=0.5pt, black] (1,-1) to (-1,-1); 
    \draw [bend right=15, line width=0.5pt, black] (1,-1) to (-1,-1);

    \draw [bend left=15, line width=0.5pt, black] (-1,1) to (-1,-1); 
    \draw [bend right=15, line width=0.5pt, black] (-1,1) to (-1,-1);
    \draw [line width=0.5pt, black] (-1,1) to (-1,-1);
    
    \draw [fill=white,line width=0.5pt] (-1,1) node[left] {} circle (3.5pt) ; 
    \draw [fill=white,line width=0.5pt] (1,1) node[right] {} circle (3.5pt) ; 
    \draw [fill=white,line width=0.5pt] (-1,-1) node[left] {} circle (3.5pt) ; 
    \draw [fill=white,line width=0.5pt] (1,-1) node[right] {} circle (3.5pt) ;
\end{tikzpicture}
        \caption{$Q_{2,2,2,3}$}
		\label{fig:Q2223}
    \end{subfigure}
    \caption{Four examples for $Q_{a_1,a_2,a_3,a_4}$.}
    \label{fig:4-graphs-Qabcd}
\end{figure}
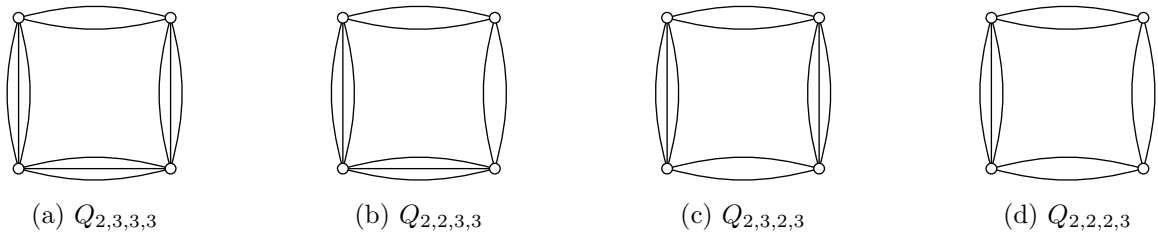

\begin{corollary}\label{cor:Qabcd-forbidden}
Let $H$ be a subgraph of $G$. The following two statements hold.
\begin{enumerate}[label=(\arabic*)]
\setlength{\itemsep}{0em}
\item\label{Q2333}
If $H =Q_{2,3,3,3}$ and there exists an $(H,n)$-path $P_n = v_0 v_1 \dots v_n$ in $G$, then $n \ge 4$.

\item\label{Q2233} 
If $H \in \{Q_{2,2,3,3},Q_{2,3,2,3}\}$, and there exists an $(H,2)$-path $P = u_0 u_1 u_2$ in $G$, then for 
any other internally disjoint $(H,n)$-path $Q = v_0 v_1 \dots v_n$ in $G$, we have $n \ge 4$.
\end{enumerate}
\end{corollary}

\begin{proof}
Since $G$ contains no $T_{1,1,3}$ and $\mu(G)\leq 3$, $H$ is an induced subgraph. Note that $w(Q_{2,3,3,3}) = 2 \times 11 - 10 \times 4 + 16 = -2$, and
$w(Q_{2,2,3,3}) = w(Q_{2,3,2,3}) = 2 \times 10 - 10 \times 4 + 16 = -4$.
So we are done by \Cref{lem:w=-2-path} and \Cref{lem:w=-4-path}.
\end{proof}

\begin{lemma}\label{lemma:Q2223-forbidden}
Let $H$ be a copy in $G$ of $Q_{2,2,2,3}$, with $V(H) = \{v_1, v_2, v_3, v_4\}$. If $\mu_H(v_1,v_2) = \mu_H(v_2,v_3) = \mu_H(v_3,v_4) = 2$, and $\mu_H(v_4,v_1) = 3$, then there do not exist distinct vertices $x,y,z$ in $V(G)\setminus V(H)$ such that $v_1xv_2$, $v_2yv_3$, and $v_3zv_4$ are 
internally disjoint $(H,2)$-paths and $\mu_G(v_2,y) = \mu_G(y,v_3) = 2$, as shown in \Cref{fig:Q2223-forbidden}.
\end{lemma}

\begin{figure}[htbp!]
    \centering
    \begin{tikzpicture}[scale=0.8]			
    \draw [bend left=15, line width=0.5pt, black] (-1,1) to (1,1); 
    \draw [bend right=15, line width=0.5pt, black] (-1,1) to (1,1);

    \draw [bend left=15, line width=0.5pt, black] (1,-1) to (1,1); 
    \draw [bend right=15, line width=0.5pt, black] (1,-1) to (1,1);

    \draw [bend left=15, line width=0.5pt, black] (1,-1) to (-1,-1); 
    \draw [bend right=15, line width=0.5pt, black] (1,-1) to (-1,-1);
    \draw [line width=0.5pt, black] (1,-1) to (-1,-1);

    \draw [bend left=15, line width=0.5pt, black] (-1,1) to (-1,-1); 
    \draw [bend right=15, line width=0.5pt, black] (-1,1) to (-1,-1);

    \draw [line width=0.5pt, black] (-1,1) to (-2.732,0);
    \draw [line width=0.5pt, black] (-1,-1) to (-2.732,0);

    \draw [line width=0.5pt, black] (1,-1) to (2.732,0);
    \draw [line width=0.5pt, black] (1,1) to (2.732,0);

    \draw [bend left=15,line width=0.5pt, black] (-1,1) to (0,2.732);
    \draw [bend right=15,line width=0.5pt, black] (-1,1) to (0,2.732);
    \draw [bend left=15,line width=0.5pt, black] (1,1) to (0,2.732);
    \draw [bend right=15,line width=0.5pt, black] (1,1) to (0,2.732);
    
    \draw [fill=white,line width=0.5pt] (-1,1) node[left=1.1mm] {$v_2$} circle (3.5pt) ; 
    \draw [fill=white,line width=0.5pt] (1,1) node[right=1.1mm] {$v_3$} circle (3.5pt) ; 
    \draw [fill=white,line width=0.5pt] (-1,-1) node[left=1.1mm] {$v_1$} circle (3.5pt) ; 
    \draw [fill=white,line width=0.5pt] (1,-1) node[right=1.1mm] {$v_4$} circle (3.5pt) ;
    \draw [fill=white,line width=0.5pt] (-2.732,0) node[left] {$x$} circle (3.5pt) ; 
    \draw [fill=white,line width=0.5pt] (2.732,0) node[right] {$z$} circle (3.5pt) ;
    \draw [fill=white,line width=0.5pt] (0,2.732) node[left] {$y$} circle (3.5pt) ;
    
\end{tikzpicture}

\caption{Forbidden subgraph when $H=Q_{2,2,2,3}$.}
\label{fig:Q2223-forbidden}
\end{figure}
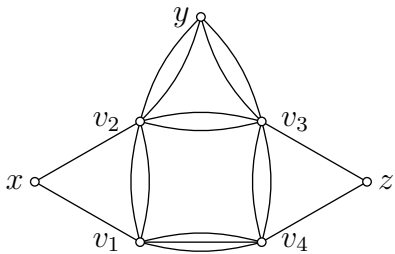

\begin{proof}
Note that $H$ is an induced subgraph, as $G$ does not contain $T_{1,1,3}$ and $\mu(G)\leq 3$. Suppose instead that $3$ such vertices exist. Let $P_1:=v_1xv_2$, $P_2:=v_2yv_3$, and $P_3:=v_3zv_4$.
Lift the paths $P_1$, $P_2$, and $P_3$. The induced subgraph on $V(H)$ then becomes $Q_{3,3,3,3}$, which is $\S_5$-contractible by \Cref{thm:small-S5-contractible}\ref{3_thm:small-S5-contractible}. Contracting this subgraph into a single vertex $v_H$ yields a graph $G'$. Since $G$ is $7$-edge-connected by \Cref{lemma:7-edge-conn} and essentially $10$-edge-connected by \Cref{lemma:essential-10}, it follows that $G'$ is 4-edge-connected.

Let $\mathcal{P}$ be any partition of $G'$, say with parts $V_1, \ldots, V_t$.
Let $\mathcal{P}_H$ be the corresponding $H$-restored partition of $G$, say with parts $V_1', \ldots, V_t'$, 
where $V(H) \subseteq V_1'$. Since $v(H) = 4$, the type of $\mathcal{P}_H$ is $(4^+, *)$. Note that $w(H) = 2 \times 9 - 10 \times 4 + 16 = -6$, and $H$ must be induced; otherwise, $G$ contains $T_{1,1,3}$ or $4K_2$, a contradiction.

For the trivial partition $\P$ of $G'$, we have 
    \[
        w_{G'}(\mathcal{P}) = w_G(\mathcal{P}_H) - 2 \times 6 = (6 - (-6)) - 6 \times 2 = 0.
    \]  
For each partition $\P$ of type $(4^+, 2^+, 2^+, *)$, we have 
    \[
        w_{G'}(\mathcal{P}) \ge w_G(\mathcal{P}_H) - 2 \times 6 \ge (8 + 4 + 4) - 6 \times 2 = 4.
    \]  

Next, consider the case where $\mathcal{P}_H$ has type $(4^+, 2^+, 1,\dots,1)$ or $(5^+, 1,\dots,1)$. 
There are $4$ possibilities based on $V_1'$:

    \noindent
    (a) If $V_1' = V(H)$, then $\sigma(G[V_1']) = 6 - w(H) = 12$, so $\mathcal{P}_H$ has type $(4,2^+,1,1,\dots,1)$. Since at most $6$ edges are removed in constructing $Q_{3,3,3,3}$, we have 
    \[
        w_{G'}(\mathcal{P}) \ge w_G(\mathcal{P}_H) - 2 \times 6 \ge (12 + 4) - 12 = 4.
    \]  
    Otherwise, $|V_1'| \ge 5$ and $w_G(\mathcal{P}_H) \ge \sigma(G[V_1']) \ge 8$.

    \noindent
    (b) If $|\{x, y, z\}\cap V_1'|\ge 2$, then at most $2$ edges are removed, so
    \[
        w_{G'}(\mathcal{P}) \ge w_G(\mathcal{P}_H) - 2 \times 2 \ge 8 - 4 = 4.
    \]

    \noindent
    (c) If $|\{x, y, z\}\cap V_1'| =1$, then at most $4$ edges are removed. Two of $P_1,P_2,P_3$ are $(G[V_1'],2)$-paths, so by \Cref{lem:w=-2-path,lem:w=-4-path} we have $w(G[V_1']) \le -6$. Hence, $w_G(\mathcal{P}_H) \ge \sigma(G[V_1']) \ge 12$, and
    \[
        w_{G'}(\mathcal{P}) \ge w_G(\mathcal{P}_H) - 2 \times 4 \ge 12 - 8 = 4.
    \]

    \noindent
    (d) If $|\{x, y, z\}\cap V_1'|=0$, then similarly $w(G[V_1']) \le -6$ and $\sigma(G[V_1']) \ge 12$. If $\mathcal{P}_H$ has type $(4^+, 2^+, 1,\dots,1)$, then
    \[
        w_{G'}(\mathcal{P}) \ge (12 + 4) - 12 = 4,
    \]
    and if it has type $(5^+, 1,\dots,1)$, then
    \[
        w_{G'}(\mathcal{P}) \ge 12 - 12 = 0.
    \]

    In all cases, we have $w(G') \ge 0$. By \Cref{lem:property}\ref{lem:property-lifting}, $G'$ is not $\S_5$-contractible. Therefore, $G$ has a partition $\mathcal{P}'$ such that $G'/\mathcal{P}' \in \mathcal{N}_5$. Since $G'$ is $4$-edge-connected, we have $G'/\mathcal{P}' \in \mathcal{N}_5 \setminus \{2K_2,3K_2\}$, and $w_{G'}(\mathcal{P}') = 0$. The analysis above implies that $V(H) \subseteq V_1'$, $|V_2'| = |V_3'| = \dots = |V_t'| = 1$, and none of $\{x, y, z\}$ lies in $V_1'$. Thus, $v(G'/\mathcal{P}') \ge 4$, and the only possibility is that $G'/\mathcal{P}' \in \{W_1, W_2\}$, where one vertex corresponds to $G[V_1']$ and the others correspond to $x, y, z$. Since both $W_1$ and $W_2$ contain a copy of $K_4$, vertices $x, y, z$ are pairwise adjacent in $G'/\mathcal{P}'$ and hence also in $G$. However, the induced subgraph $G[\{x, y, z, v_2, v_3\}]$ then contains a copy $H'$ of $T_{2,2,2}$ with $2$ internally disjoint $(H',2)$-paths $v_2 x y$ and $v_3 z y$, contradicting \Cref{lemma:T222-forbidden}. 

    This completes the proof.
\end{proof}

\subsection{Discharging}

In what follows, we slightly abuse notation, by writing $G$ for both a planar graph and one of its fixed planar embeddings. Let $F(G)$ denote the set of faces of $G$. For each face $f \in F(G)$, let $H_f$ denote the subgraph consisting of all vertices and edges on the boundary of $f$, including all parallel copies of each boundary edge. Recall the forbidden configurations from 
\Cref{lemma:no-T113}, 
\Cref{lemma:T222-forbidden}, 
\Cref{cor:Qabcd-forbidden}, 
and \Cref{lemma:Q2223-forbidden}.
These results imply the following corollary, which we will use to analyze our discharging.

\begin{corollary}\label{cor:planar-forbidden}
Let $f$ be a face of $F(G)$.
\begin{enumerate}[label=(\arabic*)]
\setlength{\itemsep}{0em}
    \item\label{T222-two4faces}
    If $H_f=T_{2,2,2}$, then $f$ is weakly adjacent to at least two $4^+$-faces.
    \item\label{Q2333-5face}
    If $H_f=Q_{2,3,3,3}$, then $f$ is weakly adjacent via a $2K_2$ to a $5^+$-face.
    \item\label{Q2323-Q2233}
    If $H_f\in\{Q_{2,3,2,3},Q_{2,2,3,3}\}$ and $f$ is weakly adjacent to a $3$-face $f_1$ (via one $2K_2$) with $H_{f_1}=T_{2,2,2}$, then $f$ is weakly adjacent via the other $2K_2$ to a $5^+$-face.
    \item\label{Q2223-twoT222}
    If $H_f=Q_{2,2,2,3}$, then $f$ is weakly adjacent via copies of $2K_2$ to at most two $3$-faces whose induced subgraphs are each $T_{2,2,2}$.
\end{enumerate}
\end{corollary}

\begin{proof}
For \ref{T222-two4faces}, let $V(f):=\{x,y,z\}$. \Cref{lemma:T222-forbidden} implies that there do not exist two internally disjoint $(H_f,2)$-paths. If $f$ is weakly adjacent to two $3$-faces, then these two $3$-faces must share two common vertices. By symmetry, we may assume that the configuration shown in~\Cref{fig:H_f=T_{2,2,2}} must occur, where $uxy$ and $uxz$ are $3$-faces. Note that since $G$ contains no $T_{1,1,3}$, the multiplicities of $ux$, $uy$, and $uz$ are all at most $2$. Therefore, the vertex $x$ has degree less than $7$, and, hence, \ref{T222-two4faces} holds.

\begin{figure}[htbp!]
    \centering
    \begin{subfigure}{0.3\textwidth}
    \centering
    \begin{tikzpicture}[scale=0.8]			
    \draw [bend left=15, line width=0.5pt, black] (-1,0) to (0,1.73); 
    \draw [bend left=15, line width=0.5pt, black] (0,1.73) to (-1,0); 

    \draw [bend left=15, line width=0.5pt, black] (1,0) to (0,1.73); 
    \draw [bend left=15, line width=0.5pt, black] (0,1.73) to (1,0); 

    \draw [bend left=15, line width=0.5pt, black] (-1,0) to (1,0); 
    \draw [bend right=15, line width=0.5pt, black] (-1,0) to (1,0); 

    \draw [line width=0.5pt, black] (-2,1.73) to (-1,0); 
    \draw [line width=0.5pt, black] (-2,1.73) to (0,1.73); 

    \draw [line width=0.5pt, black] (-2,1.73) .. controls (-1,3) and (2.5,3) .. (1,0);  
    
    \draw [fill=white,line width=0.5pt] (-1,0) node[left] {$y$} circle (3.5pt) ; 
    \draw [fill=white,line width=0.5pt] (1,0) node[right] {$z$} circle (3.5pt) ;
    \draw [fill=white,line width=0.5pt] (0,1.73) node[above] {$x$} circle (3.5pt) ;
    \draw [fill=white,line width=0.5pt] (-2,1.73) node[left] {$u$} circle (3.5pt) ;
\end{tikzpicture}
\caption{$H_f=T_{2,2,2}$}
\label{fig:H_f=T_{2,2,2}}
\end{subfigure}
\hfill
\begin{subfigure}
{0.3\textwidth}
    \centering
    \begin{tikzpicture}[scale=0.7]			
    \draw [bend left=15, line width=0.5pt, black] (-1,1) to (1,1); 
    \draw [bend right=15, line width=0.5pt, black] (-1,1) to (1,1);
    \draw [line width=0.5pt, black] (-1,1) to (1,1);

    \draw [bend left=15, line width=0.5pt, black] (1,-1) to (1,1); 
    \draw [bend right=15, line width=0.5pt, black] (1,-1) to (1,1);

    \draw [bend left=15, line width=0.5pt, black] (1,-1) to (-1,-1); 
    \draw [bend right=15, line width=0.5pt, black] (1,-1) to (-1,-1);
    \draw [line width=0.5pt, black] (1,-1) to (-1,-1);

    \draw [bend left=15, line width=0.5pt, black] (-1,1) to (-1,-1); 
    \draw [bend right=15, line width=0.5pt, black] (-1,1) to (-1,-1);

    \draw [bend left=15, line width=0.5pt, black] (-1,1) to (-2.732,0);
    \draw [bend right=15, line width=0.5pt, black] (-1,1) to (-2.732,0);
    \draw [bend left=15, line width=0.5pt, black] (-1,-1) to (-2.732,0);
    \draw [bend right=15, line width=0.5pt, black] (-1,-1) to (-2.732,0);

    \draw [line width=0.5pt, black] (1,-1) to (2.5,1);
    \draw [bend left=90, line width=0.5pt, black] (-2.732,0) to (2.5,1);
    
    \draw [bend left=70, line width=0.5pt, black] (-2.732,0) to (1,1);

    \draw [fill=white,line width=0.5pt] (-1,1) node[left=1.1mm] {$v_2$} circle (3.5pt) ; 
    \draw [fill=white,line width=0.5pt] (1,1) node[right=1.1mm] {$v_3$} circle (3.5pt) ; 
    \draw [fill=white,line width=0.5pt] (-1,-1) node[left=1.1mm] {$v_1$} circle (3.5pt) ; 
    \draw [fill=white,line width=0.5pt] (1,-1) node[right=1.1mm] {$v_4$} circle (3.5pt) ;
    \draw [fill=white,line width=0.5pt] (-2.732,0) node[left] {$x$} circle (3.5pt) ; 
    \draw [fill=white,line width=0.5pt] (2.5,1) node[right] {$y$} circle (3.5pt) ;

    \node at (-1.7,0) {$f_1$};
    \node at (1.5,0.4) {$f_2$};
\end{tikzpicture}

\caption{$H_f=Q_{2,3,2,3}$}
\label{fig:Hf=Q2323}
\end{subfigure}
\hfill
\begin{subfigure}{0.3\textwidth}
    \centering
    \begin{tikzpicture}[scale=0.7]			

    \draw [bend left=15, line width=0.5pt, black] (-1,-1) to (-1,1); 
    \draw [bend right=15, line width=0.5pt, black] (-1,-1) to (-1,1);

    \draw [bend left=15, line width=0.5pt, black] (1,1) to (-1,1); 
    \draw [bend right=15, line width=0.5pt, black] (1,1) to (-1,1);

    \draw [bend left=15, line width=0.5pt, black] (1,1) to (1,-1); 
    \draw [bend right=15, line width=0.5pt, black] (1,1) to (1,-1);
    \draw [line width=0.5pt, black] (1,1) to (1,-1);

    \draw [bend left=15, line width=0.5pt, black] (-1,-1) to (1,-1); 
    \draw [bend right=15, line width=0.5pt, black] (-1,-1) to (1,-1);
    \draw [line width=0.5pt, black] (-1,-1) to (1,-1);

    \draw [bend left=15, line width=0.5pt, black] (-1,1) to (-2.732,0);
    \draw [bend right=15, line width=0.5pt, black] (-1,1) to (-2.732,0);
    \draw [bend left=15, line width=0.5pt, black] (-1,-1) to (-2.732,0);
    \draw [bend right=15, line width=0.5pt, black] (-1,-1) to (-2.732,0);

    \draw [line width=0.5pt, black] (-2.732,0) to (-2.5,1.5);
    \draw [line width=0.5pt, black] (-1,1) to (-2.5,1.5);

    \draw [line width=0.5pt, black]
    (1,1) .. controls (0,3.5) and (-4,3.5) .. (-2.732,0);

    \draw [fill=white,line width=0.5pt] 
    (-1,-1) node[left=1.1mm] {$v_1$} circle (3.5pt); 

    \draw [fill=white,line width=0.5pt] 
    (-1,1) node[left=1.1mm] {$v_2$} circle (3.5pt); 

    \draw [fill=white,line width=0.5pt] 
    (1,1) node[right=1.1mm] {$v_3$} circle (3.5pt); 

    \draw [fill=white,line width=0.5pt] 
    (1,-1) node[right=1.1mm] {$v_4$} circle (3.5pt); 

    \draw [fill=white,line width=0.5pt] 
    (-2.732,0) node[left] {$x$} circle (3.5pt); 

    \draw [fill=white,line width=0.5pt] 
    (-2.5,1.5) node[above] {$y$} circle (3.5pt);

    \node at (-1.7,0) {$f_1$};
    \node at (0,1.6) {$f_2$};

    \end{tikzpicture}

\caption{$H_f=Q_{2,2,3,3}$}
\label{fig:Hf=Q2233}
\end{subfigure}

\caption{Configurations in the proof of~\Cref{cor:planar-forbidden}. }
\label{fig:placeholder}
\end{figure}

For~\ref{Q2333-5face}, it follows directly from \Cref{cor:Qabcd-forbidden}\ref{Q2333}.

For \ref{Q2323-Q2233}, let $V(f):=\{v_1,v_2,v_3,v_4\}$. Let $f_1$ be the $3$-face with $H_{f_1}=T_{2,2,2}$ specified in the statement with $V(f_1):=\{v_1,v_2,x\}$, and let $f_2$ be the $3^+$-face weakly adjacent to $f$ via the other $2K_2$ in $G$. Assume to the contrary that $f_2$ is a $4^-$-face. Then \Cref{cor:Qabcd-forbidden}\ref{Q2233} implies that $f_1$ and $f_2$ must share additional vertices in addition to those of $f$. 

First we consider the case when $H_f=Q_{2,3,2,3}$ with $\mu(v_1,v_2)=\mu(v_3,v_4)=2$. Since $f_2$ is a $4^-$-face, at least one of $v_3$ and $v_4$ must be adjacent to $x$, as shown in \Cref{fig:Hf=Q2323}, which implies that $G$ contains a $T_{1,1,3}$, a contradiction.

Now consider the case when $H_f=Q_{2,2,3,3}$ with $\mu(v_1,v_2)=\mu(v_2,v_3)=2$. If $f_2$ is a $3$-face, then $V(f_2)=\{v_2,v_3,x\}$. Since $G$ contains no $T_{1,1,3}$, it implies that $v_2$ has degree less than 7, contradicting~\Cref{lemma:7-edge-conn}.
Hence, we may assume that $f_2$ is a $4$-face. Let $V(f_2)=\{v_2,v_3,x,y\}$. Since $\delta(G)\ge 7$ and $G$ contains no $T_{1,1,3}$, the boundary walk of $f_2$ must be $v_2v_3xy$ (up to the choice of direction), rather than $v_2v_3yx$, as shown in \Cref{fig:Hf=Q2233}. Note that $xyv_2$ may not be a face. We then observe that the paths $xyv_2$ and $xv_3v_2$ are two internally disjoint $(H_{f_1},2)$-paths, contradicting~\Cref{lemma:T222-forbidden}, and hence it completes the proof.

For \ref{Q2223-twoT222}, let $V(f)=\{v_1,v_2,v_3,v_4\}$, where $\mu(v_1,v_4)=3$. Assume that $f$ is weakly adjacent, via three $2K_2$'s, to three $3$-faces $f_1,f_2$, and $f_3$ with $V(f_1)=v_1xv_2$, $V(f_2)=v_2yv_3$, and $V(f_3)=v_3zv_4$, each of which induces a $T_{2,2,2}$. Then, by \Cref{lemma:Q2223-forbidden}, some of $x,y,z$ must coincide. First, $x=z$ is impossible; otherwise, $v_1xv_4$ induces a $T_{1,1,3}$. Hence, we must have either $y=x$ or $y=z$. The former case implies that $d(v_2)=6$, while the latter one implies that $d(v_3)=6$, both of which contradict the $7$-edge-connectivity of $G$. Therefore, \ref{Q2223-twoT222} holds.
\end{proof}

Now we assign to each face $f \in F(G)$ an initial charge
$ch_0(f) := d(f) - \frac{5}{2}.$
Recall that $w(G)=0$.
By the definition of the weight function, $0 = w(G) = 2e(G) - 10v(G) + 16$, so $0 = e(G)/2-5v(G)/2+4$.
And Euler’s formula gives $f(G) - e(G) + v(G) = 2$, so $f(G)=e(G)-v(G)+2$.  By substituting these
two equalities into the sum below, we get

\begin{align*}
\sum_{f \in F(G)} ch_0(f)
	&= \sum_{f \in F(G)} \left(d(f) - \frac{5}{2}\right) = 2e(G) - \frac{5}{2} f(G)\\
	&= 2e(G) - \frac{5}{2} \left(e(G)-v(G)+2\right)
	=\frac52v(G)-\frac12e(G)-5 = 4-5 < 0.
\end{align*}

We now redistribute charge via the following rules.
Let $ch_1(f)$ denote the charge of $f$ after applying rule (R1), and let $ch_2(f)$ denote the final charge after applying rule (R2).

\begin{enumerate}[label=\textbf{(R\arabic*)}]
\setlength{\itemsep}{0em}
    \item Every $3^+$-face sends $\frac{1}{4}$ to each weakly adjacent $2$-face.
    \item Let $f_1$ and $f_2$ be two $3^+$-faces that are weakly adjacent via $2K_2$.
    \begin{enumerate}
    \setlength{\itemsep}{0em}
        \item[(R2.1)] If $f_1$ is a $5^+$-face, then $f_1$ sends $\frac{1}{4}$ to $f_2$.
        \item[(R2.2)] If $f_1$ is a $4$-face with $H_{f_1}=Q_{a,b,c,d}$ and $\min\{a,b,c,d\}=1$, then $f_1$ sends $\frac{1}{8}$ to $f_2$.
        \item[(R2.3)] If $f_1$ is a $4$-face with $H_{f_1}=Q_{a,b,c,d}$, $\min\{a,b,c,d\}\ge2$, and $H_{f_2}=T_{2,2,2}$, then $f_1$ sends $\frac{1}{8}$ to $f_2$.
    \end{enumerate}
\end{enumerate}

We show that every face $f$ has non-negative final charge, that is, $ch_2(f) \ge 0$.
Thus $\sum_{f \in F(G)} ch_2(f) \ge 0> \sum_{f \in F(G)} ch_0(f)$, and this contradiction completes the proof.

For each $3^+$-face $f$, let $wt(f)$ denote the number of $2$-faces weakly adjacent to $f$.

\medskip
\noindent\textbf{2-faces.}
Let $f$ be a $2$-face.
By (R1), $f$ receives $\frac{1}{4}$ from each of its two weakly adjacent $3^+$-faces and sends no charge.
Thus
$
ch_2(f) = ch_1(f) = ch_0(f) + 2 \times \frac{1}{4}
= -\frac{1}{2} + \frac{1}{2} = 0.
$

\medskip
\noindent\textbf{3-faces.}
Let $f$ be a $3$-face.
By \Cref{lemma:no-T113}, the graph $G$ contains no $T_{1,1,3}$, and hence $wt(f) \le 3$.
After applying (R1),
$
ch_1(f) = d(f) - \frac{5}{2} - \frac{1}{4} wt(f).
$
If $wt(f) \le 2$, then $ch_2(f) \ge ch_1(f) \ge 0$.
If $wt(f) = 3$, then $f = T_{2,2,2}$.
By \Cref{cor:planar-forbidden}\ref{T222-two4faces}, $f$ is weakly adjacent to at least two $4^+$-faces.
By (R2.3),
$
ch_2(f) \ge ch_1(f) + 2 \times \frac{1}{8} \ge 0.
$

\medskip
\noindent\textbf{4-faces.}
Let $f$ be a $4$-face.
Then $H_f= Q_{a,b,c,d}$ with $\min\{a,b,c,d\} \ge 1$.
Since $G$ contains no $4K_2$, we have $\max\{a,b,c,d\} \le 3$.
Moreover, as $G$ contains no $Q_{3,3,3,3}$ (which is $\S_5$-contractible), it follows that $a+b+c+d \le 11$.
By (R1),
\[
ch_1(f)
= d(f) - \frac{5}{2} - \frac{1}{4}(a+b+c+d-4)
= \frac{1}{4}\bigl(10-(a+b+c+d)\bigr).
\]

If $\min\{a,b,c,d\}=1$, then (R1) and possibly (R2.2) 
give
$ch_2(f) \ge ch_0(f) - 3 \times (\frac14+\frac14) = 0$.
So we assume $\min\{a,b,c,d\} \ge 2$, and thus $a+b+c+d \in \{8,9,10,11\}$.

\begin{itemize}
\setlength{\itemsep}{0em}
    \item If $H_f = Q_{2,3,3,3}$, then $ch_1(f) = -\frac{1}{4}$.
    By \Cref{cor:planar-forbidden}\ref{Q2333-5face}, $f$ is weakly adjacent to a $5^+$-face via $2K_2$.
    Hence, by (R2.1), $ch_2(f) \ge ch_1(f) + \frac{1}{4} \ge 0$.
    \item If $H_f \in\{Q_{2,2,3,3},Q_{2,3,2,3}\}$, then $ch_1(f) = 0$.
    If $f$ is weakly adjacent to a $3$-face $f'$ with $H_{f'}=T_{2,2,2}$, then by \Cref{cor:planar-forbidden}\ref{Q2323-Q2233} it is also weakly adjacent to a $5^+$-face via $2K_2$.
    Therefore, by (R2.3) and (R2.1),
    $
    ch_2(f) \ge ch_1(f) - \frac{1}{8} + \frac{1}{4} = \frac{1}{8} > 0.
    $
    Otherwise, $ch_2(f) \ge ch_1(f) = 0$.
    \item If $H_f = Q_{2,2,2,3}$, then $ch_1(f) = \frac{1}{4}$.
    By \Cref{cor:planar-forbidden}\ref{Q2223-twoT222}, $f$ is weakly adjacent via copies of $2K_2$ to at most two copies of $T_{2,2,2}$.
    By (R2.3),
    $
    ch_2(f) \ge ch_1(f) - 2 \times \frac{1}{8} = 0.
    $
    \item If $H_f = Q_{2,2,2,2}$, then $ch_1(f) = \frac{1}{2}$.
    Recall that $f$ is weakly adjacent via copies of $2K_2$ to at most four copies of $T_{2,2,2}$, so (R2.3) implies
    $
    ch_2(f) \ge ch_1(f) - 4 \times \frac{1}{8} = 0.
    $
\end{itemize}

\noindent\textbf{$\bm{5^+}$-faces.}
Let $f$ be a $5^+$-face.
Then
$ch_2(f) \ge ch_0(f) - 2 \times \frac{1}{4} d(f)
= \frac{1}{2}\bigl(d(f) - 5\bigr) \ge 0.$

This completes the proof of Structural Theorem. 

\section*{Acknowledgments}
Thanks to two anonymous referees, whose careful reading and numerous suggestions improved the clarity of this paper.

\end{document}